\documentclass[sn-mathphys-num]{sn-jnl}
\usepackage{graphicx}%
\usepackage{multirow,longtable}%
\usepackage{amsmath,amssymb,amsfonts}%
\usepackage{amsthm}%
\usepackage{mathrsfs}%
\usepackage[title]{appendix}%
\usepackage{xcolor}%
\usepackage{textcomp}%
\usepackage{manyfoot}%
\usepackage{booktabs}%
\usepackage{algorithm}%
\usepackage{algorithmicx}%
\usepackage{algpseudocode}%
\usepackage{listings}%
\usepackage{tabularx} 
\usepackage{geometry}

\renewcommand{\textcolor}[2]{#2}
\theoremstyle{thmstyleone}%
\newtheorem{theorem}{Theorem}
\newtheorem{lemma}{Lemma}%
\theoremstyle{thmstyletwo}%
\theoremstyle{thmstylethree}%
\begin{document}

\title{A scalable sequential adaptive cubic regularization algorithm for optimization with general equality constraints}


\author*[1]{\fnm{Yonggang} \sur{Pei}}\email{peiyg@htu.edu.cn}
\author[1]{\fnm{Yubing } \sur{Lin}}\email{linyb17838409600@163.com}
\equalcont{These authors contributed equally to this work.}
\author[1]{\fnm{Shuai} \sur{Shao}}\email{shaoshuai20000525@163.com}
\equalcont{These authors contributed equally to this work.}
\author[2]{\fnm{Mauricio Silva} \sur{Louzeiro}}\email{mauriciolouzeiro@ufg.br}
\equalcont{These authors contributed equally to this work.}
\author[3]{\fnm{Detong} \sur{Zhu}}\email{dtzhu@shnu.edu.cn}
\equalcont{These authors contributed equally to this work.}


\affil*[1]{\orgdiv{School of Mathematics and Statistics}, \orgname{Henan Normal University}, \city{Xinxiang},\orgaddress{\street{} \postcode{453007},~\state{}\country{China}}}
\affil[2]{\orgdiv{Instituto de Matem\'{a}tica e Estat\'{i}stica}, \orgname{Universidade Federal de Goi\'{a}s}, \city{Avenida Esperan\c{c}a},\orgaddress{\street{} \postcode{74690-900},~\state{}\country{Brazil}}}
\affil[3]{\orgdiv{Mathematics and Science College}, \orgname{Shanghai Normal University},\orgaddress{\street{} \city{Shanghai}, \postcode{200234},~\state{}\country{China}}}







\abstract{The scalable adaptive cubic regularization method ($\mathrm{ARC_{q}K}$: Dussault et al. in Math. Program. Ser. A  207(1-2): 191-225, 2024) has been recently proposed for unconstrained optimization. It has excellent convergence properties, well-defined complexity bounds, and promising numerical performance. In this paper, we extend $\mathrm{ARC_{q}K}$ to nonlinear optimization with general equality constraints and propose a scalable sequential adaptive cubic regularization algorithm named $\mathrm{SSARC_{q}K}$. In each iteration, we construct an ARC subproblem with linearized constraints inspired by sequential quadratic optimization methods. Next, a composite-step approach is used to decompose the trial step into the sum of a vertical step and a horizontal step. By means of the reduced-Hessian approach, we rewrite the linearly constrained ARC subproblem as a standard unconstrained ARC subproblem to compute the horizontal step. \textcolor{blue}{Analogous to $\mathrm{ARC_{q}K}$}, we employ a CG-Lanczos procedure with shifts to solve ARC subproblems inexactly, \textcolor{blue}{thus bypassing any hard case consideration. This also avoids solving the subproblem multiple times for obtaining a new iterative point}. We establish the global convergence of the inexact ARC method \textcolor{blue}{$\mathrm{SSARC_{q}K}$ to first-order critical points.} Preliminary numerical tests and \textcolor{blue}{some comparison results} are presented to illustrate the performance of $\mathrm{SSARC_{q}K}$.}
 




\maketitle

\section{Introduction}\label{sec1}
\textcolor{blue}{Nonlinear optimization is a subject that is widely and increasingly used in science, engineering, economics, management, industry, and other areas \cite{sunyuanbook2006,Nocedal1999}. Many impressive methods have been developed for solving nonlinear optimization problems. Most of these methods fall into two categories: trust region methods and line search methods (see, e.g., \cite{Nocedal1999,Conn15,Dennis45,Dennis1997SIOPT,Byrd2010MP,Heinkenschloss2014SIOPT,Bergou2021SIAMJSC,Fischer2024COAP,Onwunta2024JSC}). The Adaptive Regularization method with Cubic (ARC), recently explored by Cartis et al. \cite{Cartis11,Cartis12} for solving unconstrained optimization, has been shown to possess excellent global and local convergence properties as the third alternative method. Research on ARC methods has developed rapidly, and a series of ARC-type (or its variants) algorithms have been proposed for unconstrained optimization (see, e.g., \cite{Bellavia41,Benson42,Bergou4,Bianconcini5,Bianconcini6,Birgin7,Cartis43,Martnez2017JGO,Dussault18,Park2020JOTA,2020LiuHongweiNA,Dehghan44,Zhao40}).
For constrained optimization problems, many researchers have also proposed some effective algorithms based on the ARC framework (see, e.g., \cite{Agarwal1,Cartis46,Zhu47} and the related references in \cite{Pei32,Pei33}).} 

\textcolor{blue}{Despite the strong theoretical guarantees of ARC methods mentioned above, their practical performance depends critically on the efficiency of solving cubic regularization subproblems and the additional implementation complexity is needed. Although there are some effective algorithms (see, e.g., \cite{Hsia2017OMS,LiederSubproblem2020,JiangRujun2021,shen2022SMAA}), most of them still need to address the intractable hard-cases like in trust region subproblem in algorithm design. Moreover, to obtain the new iterative point, these methods are likely to require solving the subproblem repeatedly, thus increasing the computational burden of the algorithm. Dussault et al. \cite{Dussault2019INFORMS,Dussault19} presented a scalable adaptive cubic regularization method ($\mathrm{ARC_{q}K}$), where ARC subproblems are solved bypassing any hard case consideration. It mainly solves a set of shifted systems concurrently through an appropriate modification of the Lanczos formulation of the conjugate gradient method, and all shifted systems are solved inexactly. This avoids solving the subproblem repeatedly.  Theory analysis is conducted on its worst-case evaluation complexity, global convergence, and asymptotic convergence \cite{Dussault19}.}

Motivated by the above methods, we aim to extend the $\mathrm{ARC_{q}K}$ method to solve nonlinear optimization with general equality constraints, and propose a scalable sequential adaptive cubic regularization algorithm, which we name $\mathrm{SSARC_{q}K}$. Similar to sequential quadratic optimization methods, we first consider a cubic regularization subproblem with linearized constraints in each iteration. Then, we employ a composite-step approach to handle linearized constraints, where the trial step is decomposed into the sum of the vertical step and the horizontal step in each iteration. The vertical step is used to reduce the constraint violation, and the horizontal step is used to provide sufficient reduction of the model. By reduced-Hessian approach, we construct a standard ARC subproblem to obtain the horizontal step. \textcolor{blue}{We use CG-Lanczos with shifts method, which is tailored from Algorithm 2 in \cite{Dussault19} to solve the ARC subproblem inexactly.} $m$ shift parameters are chosen to compute the horizontal steps, followed by the trial steps. When an iteration is unsuccessful, our algorithm does not need to recompute the trial step, but rather simply turns its attention to the next shift value and the trial step corresponding to the next shift value. We employ the exact penalty function as the merit function. The ratio of the penalty function reduction to the quadratic model reduction is used to assess the acceptance of the trial point. Global convergence \textcolor{blue}{to first-order critical points} is proved under some suitable assumptions.

The remainder of this paper is organized as follows. In Section \ref{sec2}, we describe the step computation and the $\mathrm{SSARC_{q}K}$ algorithm for equality constrained optimization. The global convergence to first-order critical points is proved in Section \ref{sec3}.
Preliminary numerical results and some comparison results are reported in Section \ref{sec4}, and the conclusion is presented in Section \ref{sec6}.

Throughout the paper, $\|\cdot\|$ denotes the Euclidean norm. The inner product of vectors $a,b\in\mathbb{R}^n$ is denoted by $a^T b$ or $\langle a,b \rangle$. 

\section{Development of the algorithm $\mathrm{SSARC_{q}K}$}\label{sec2}
Consider the following nonlinear equality constrained optimization problem
\begin{subequations}\label{NLP}
\begin{align}
{\underset {x\in\mathbb{R}^n}{{\rm minimize}}}   &\ \ \  f(x) \ \\
\mbox{subject to} &\ \ \  c(x)= 0,
\end{align}
\end{subequations}
where $f: \mathbb{R}^n \rightarrow \mathbb{R}$ and $c: \mathbb{R}^n \rightarrow \mathbb{R}^p$ are sufficiently smooth functions and \textcolor{blue}{$p\leq n$.}

For this problem, we define its Lagrangian function as
\begin{eqnarray*}
 L(x,s):=f(x)-s^Tc(x),
\end{eqnarray*}
where $s$ is the vector of Lagrange multipliers.\\
We use $J(x)$ to denote the Jacobian matrix of $c(x)$ at $x$, that is,
\begin{equation*}
  J(x)^T=[\nabla c_1(x), \nabla c_2(x), \cdot\cdot\cdot , \nabla c_p(x)],
\end{equation*}
where $c_i(x)$ $(i=1,\cdot\cdot\cdot,p)$ is the $i$th component of $c(x)$. Additionally, $g(x)$ denotes the gradient of $f$ at $x$.


Inspired by SQP methods \textcolor{blue}{(see, e.g., \cite{Wilson1963SQP,Han1976MP,Byrd1987SIAMna,Boggs1995ActaNumer,Gould2005ActaNumer,Gill2005SIAMRev}),}
we consider the following cubic regularization subproblem with linearized constraints in each iteration:
\begin{subequations}\label{constrained subp}
\begin{align}
{\underset {d\in\mathbb{R}^n}{\mbox{minimize}}} &\ \ \  f_k+g_k^Td+\frac{1}{2}d^TB_kd+\frac{1}{3}\beta_k^{-1}\|d\|^3  \\
\mbox{subject to} &\ \ \  J_kd+c_k=0,
\end{align}
\end{subequations}
where $f_k:=f(x_k)$, $g_k:=\nabla f(x_k)$, $J_k:=J(x_k)$, $c_k:=c(x_k)$, $B_k$ denotes the Hessian $\nabla_{xx}L(x_k,s_k)$ or its approximation, and ${\beta _k} \in {\mathbb{R}^ + } = \left\{ {a \in \mathbb{R}\left| {a > 0} \right.} \right\}$ is an adaptive parameter. Since the linearized constraints may be inconsistent, we assume that ${J_k}$ has full row rank for all $k$.

We do not directly solve subproblem \eqref{constrained subp} but instead adopt the composite-step approach. In each iteration $k$, we decompose the overall trial step $d_k$ into a vertical step $v_k$ and a horizontal step $h_k$, that is,
\begin{equation}\label{dkdef}
  d_k:=v_k+h_k,
\end{equation}
where the role of the vertical step  $v_k$  is to improve feasibility, while the role of the horizontal step $h_k$ is to achieve a sufficient reduction in the model of the objective function.

Since $h_k$ is used to ensure the sufficient decrease in the objective function's model, the maximum size of $v_k$ should be restricted. It is generated as follows.
\begin{eqnarray}\label{nk}
  v_k=\alpha_kv_k^c,
\end{eqnarray}
where
\begin{eqnarray*}
v_k^c=-J_k^T(J_kJ_k^T)^{-1}c_k
\end{eqnarray*}
and
\begin{eqnarray}\label{aa1h}
\alpha_k\in
\left[{\rm min}\left\{1,\frac{\theta\sqrt{\beta _k}}{\|v_k^c\|}\right\},
{\rm min}\left\{1,\frac{\sqrt{\beta _k}}{\|v_k^c\|}\right\}
\right],~~~{\rm for~}\theta\in (0,1].
\end{eqnarray}
Then, the horizontal step $h_k$ is computed as
\begin{eqnarray}\label{tn}
  h_k=Z_ku_k,
\end{eqnarray}
\textcolor{blue}{where the columns of $Z_k$ form an orthonormal basis for the null space of $J_k$ (so that $J_k Z_k=0$)} and $u_k$ is the solution (or its approximation) of the following standard ARC subproblem
\begin{equation}\label{tk}
{\underset {u}{\mbox{minimize}}}\ \ \  \langle g^Z_k,u\rangle+\frac{1}{2}\langle u,B^Z_ku\rangle+\frac{1}{3}\beta_k^{ - 1}\|u\|^3,
\end{equation}
where 
\begin{equation}\label{Bzkgzkdef}
B^Z_k:=Z^T_kB_kZ_k, g^Z_k:=Z^T_k(g_k+B_kv_k).
\end{equation}

Then, by \cite{Cartis11}, $u_k$ is a global minimizer of \eqref{tk} if and only if
\begin{subequations}\label{TKM}
\begin{align}
g^Z_k+(B^Z_k+\lambda_kI)u_k &=0,\\
B^Z_k+\lambda_kI &\succeq0,\\
\lambda_k &=\|u_k\|/\beta_k.\label{7c}
\end{align}
\end{subequations}
\textcolor{blue}{We do not solve \eqref{tk} by \eqref{TKM} directly, but obtain an approximate minimizer of \eqref{tk} by solving \eqref{TKm}, which is an approximation of \eqref{TKM} as follows.}
\begin{subequations}\label{TKm}
\begin{align}
g^Z_k+(B^Z_k+\lambda_kI)u_k &=r_k,\label{8a}\\
u_k^T(B^Z_k+\lambda_kI)u_k &\geq0,\label{8b}\\
\frac{1}{t}\frac{\|u_k\|}{\beta_k}\leq\lambda_k &\leq t\frac{\|u_k\|}{\beta_k},\label{8c}
\end{align}
\end{subequations}
where $r_k$ is a residual and $t\geq 1$ is a sampling parameter.

Algorithm 2 in \cite{Dussault19} presents the implementation of CG-Lanczos with shifts for a generic symmetric system \textcolor{blue}{$Ax=b$} with shifts $\lambda_i$, that is
$$\textcolor{blue}{(A+\lambda_iI)x=b, ~~~i=0,...,m.}$$
 We adapt Algorithm 2 in \cite{Dussault19} to solve \eqref{TKm}.

 Consistent with Algorithm 2 in \cite{Dussault19}, we select shift parameters  $\lambda_i$ to satisfy 
 \begin{equation}\label{shiftbd}
 10^{-e_L}\leq\lambda_i\leq10^{e_U}, i=0,\ldots,m,
 \end{equation} 
where $e_L,e_U\sim15$.  Our implementation also uses $e_L=10$ and $e_U=20$. Set $\lambda_{i+1}=\psi\lambda_i$, for $\psi=10^{\frac{1}{2}}$ and $\lambda_0=10^{-\frac{e_L}{2}}$.
Then we get the $e_L+e_U+1$ values $\lambda_i$.

 Next, we impose accuracy requirements on the residual $r_k$ in \eqref{8a} as follows.
 \begin{eqnarray}\label{rk}
 \|r_k\|\leq \xi{\min}\{\|g_k^Z\|, \|u_k\|\}^{1+\zeta},
\end{eqnarray}
where $\xi>0$ and $0<\zeta\leq1$. This is the same as the implementation of Algorithm 2 in \cite{Dussault19}.



We now present the specific algorithm as follows.
\begin{algorithm}[h]
\caption{Lanczos-CG with shifts for \eqref{TKm}}\label{Alg1}
\begin{algorithmic}[1]
\For{$i\leftarrow0,1,2,...,m$}
            \State $j=0.$
            \State ${\rm Set}~u_j=0, \mu_j\upsilon_j=-g_k^Z, p_j=-g_k^Z.$
            \State ${\rm Set}~ \upsilon_{j-1}=0, \sigma_j=\mu_j, \omega_{j-1}=0, \gamma_{j-1}=1.$ 
            \While {$\|g_k^Z+(B_k^Z+\lambda_iI)u_j\|>\xi\min\{\|g_k^Z\|, \|u_j\|\}^{1+\zeta}$}
            \State $\delta_j=\upsilon^T_jB_k^Z\upsilon_j$
            \State $\mu_{j+1}\upsilon_{j+1}=B_k^Z\upsilon_j-\delta_j\upsilon_j-\mu_j\upsilon_{j-1}$
            \State $\delta_j=\delta_j+\lambda_i$
            \State $\gamma_j=1/(\delta_j-\omega_{j-1}/\gamma_{j-1})$
            \If{$\gamma_j < 0$}
            \State ${\rm break}$;
             \EndIf
            \State $\omega_j=(\mu_{j+1}\gamma_j)^2$
            \State $\sigma_{j+1}=-\mu_{j+1}\gamma_j\sigma_j$
            \State $u_{j+1}=u_j+\gamma_jp_j$
            \State $p_{j+1}=\sigma_{j+1}\upsilon_{j+1}+\omega_jp_j$
            \State $j=j+1$
            \EndWhile
             \EndFor
\end{algorithmic}
\end{algorithm}

In Algorithm \ref{Alg1}, when $\gamma_j<0$, the solution of the $i$th system $(B_k^Z+\lambda_iI)u_j=-g_k^Z$ is interrupted, which indicates that $B_k^Z+\lambda_iI\nsucceq0$. Conversely, if $\gamma_j\geq0$, the solution of the $i$th system is terminated as soon as the residual norm $\|r_j\|=\|g_k^Z+(B_k^Z+\lambda_iI)u_j\|$ satisfies \eqref{rk}. $\mu_j>0$ such that $\|\upsilon_j\|=1$;  $\mu_{j+1}>0$ such that $\|\upsilon_{j+1}\|=1$.

Line 5 and lines 10-12 in Algorithm \ref{Alg1} are different from Algorithm 2 in \cite{Dussault19}. Adding these lines makes the algorithm's stopping criterion more explicit.


After obtaining the vertical step $v_k$ and the horizontal step $h_k$, we get the trial step $d_k$. Next, we use a merit function to determine whether the trial step $d_k$ should be accepted. Here, we use $l_2$ penalty function as the merit function, defined as follows.
\begin{equation}\label{phidef}
  \phi(x,\mu):=f(x)+\mu\|c(x)\|.
\end{equation}
Since the descent ratio in our algorithm is calculated using a quadratic model, the quadratic model of \eqref{phidef} is given by
\begin{eqnarray}\label{phimodelqdef}
   q(x,B,\mu,d):=f(x)+g(x)^Td+\frac{1}{2}d^TBd+\mu\|c(x)+J(x)d\|.
\end{eqnarray}
For convenience, define
\begin{eqnarray}
  q^F(x,B,v)&:=&f(x)+g(x)^Tv+\frac{1}{2}v^TBv,\label{qFdef} \\
  q^N(x,v)&:=&\|c(x)+J(x)v\|,\label{qNdef}\\
  q^H(x,B,h)&:=&(g(x)+Bv)^Th+\frac{1}{2}h^TBh.\label{qHdef}
\end{eqnarray}
Define the overall model decrease from $0$ to $d_k$ as
\begin{equation*}
  \Delta q_k:=q(x_k,B_k,\mu_k,0)-q(x_k,B_k,\mu_k,d_k).
\end{equation*}
Then
\begin{equation*}
\Delta q_k=\Delta q^H_k+\mu_k\Delta q^N_k+\Delta q^F_k,
\end{equation*}
where
\begin{equation}\label{deltaqNdef}
 \Delta q^N_k:=q^N(x_k,0)-q^N(x_k,v_k),
\end{equation}
\begin{equation}\label{deltaqHdef}
 \Delta q^H_k:=q^H(x_k,B_k,0)-q^H(x_k,B_k,h_k),
\end{equation}
and
\begin{equation}\label{deltaqFdef}
 \Delta q^F_k:=q^F(x_k,B_k,0)-q^F(x_k,B_k,v_k).
\end{equation}
Note that $\Delta q^N_k$ and $\Delta q^H_k$ must be positive (see Lemma \ref{lem3.1} and Lemma \ref{lem3.2}).

To obtain a significant decrease, we require that
\begin{equation}\label{deltam}
\Delta q_k =\Delta q^H_k+\mu_k\Delta q^N_k+\Delta q^F_k\geq\nu\mu_k\Delta q^N_k
\end{equation}
for some $\nu\in(0,1)$ and $\mu_k$ sufficiently large.
\textcolor{blue}{Here, we follow the rules in \cite{Vardi36} to update $\mu_k$.}
First, compute
\begin{eqnarray}\label{muc}
  \mu^c_k=-\frac{\Delta q^F_k+\Delta q^H_k}{(1-\nu)\Delta q^N_k}
\end{eqnarray}
as the smallest value satisfying  \eqref{deltam}.
Suppose that $\mu_k$ is the penalty parameter at the previous iteration.

Then $\mu_k$ is updated by
\begin{eqnarray}\label{miuk}
       \mu_k = \left\{
      \begin{array}{lll}
       {\rm max}\{\mu_k^c, \tau_1\mu_{k-1}, \mu_{k-1}+\tau_2\},&~~~{\rm if}~\mu_{k-1}<\mu_k^c,\\
       \mu_{k-1},&~~~{\rm otherwise},
      \end{array}
      \right.
\end{eqnarray}
where $\tau_1> 1,\tau_2>0$ are constants.

In order to determine whether the trial step $d_k$ is accepted, we calculate the descent ratio. The descent ratio is defined as follows.
\begin{eqnarray}\label{rhok1}
  \rho_k:=\frac{\phi(x_k,\mu_k)-\phi(x_k+d_k,\mu_k)}{q(x_k,B_k,\mu_k,0)-q(x_k,B_k,\mu_k,d_k)}.
\end{eqnarray}

Based on the above discussion, we summarize the scalable sequential adaptive cubic regularization algorithm  (Algorithm \ref{Alg2}: $\mathrm{SSARC_qK}$)  for nonlinear equality constrained optimization as follows.

\begin{algorithm}[h]
\caption{$\mathrm{SSARC_qK}$ for equality-constained optimization}\label{Alg2}
\begin{algorithmic}[1]
  \State Initialize $x_0 \in \mathbb{R}^n, \beta _0 > 0, 0 < \eta_1 < \eta _ 2< 1, 0 < \gamma _1 < 1 < \gamma _2, \mu_{ - 1} > 0, 0 < \nu  < 1, \tau _1 > 1, \tau _2 > 0,$ $\xi>0$, $0<\zeta\leq1$, $\lambda.$ Set $k=0.$
        \While{$x_k$ is not a KKT point}
            \State Compute a vertical step ${v_k}$ by \eqref{nk}.
            \State Solve $(B_k^Z + \lambda I){u}(\lambda ) =  - g_k^Z$ for ${u}(\lambda )$ by using Algorithm \ref{Alg1}.
            \State $d(\lambda) = {v_k} + {Z_k}{u}(\lambda )$.
            \State Set flag = 0.
            \State ${i^ + } = \min \left\{ {0 \le i \le m\left| {B_k^Z + {\lambda _i}I \succ 0} \right.} \right\}$.
            \State $j = \arg \min \left\{ {\left| {\beta {\lambda _i} - \left\| {{u}({\lambda _i})} \right\|} \right|\left| {{i^ + } \le i \le m} \right.} \right\}$.
        \While{flag = 0}
            \State ${d_k} = d({\lambda _j})$.
            \State Compute $\mu^c_k$ by \eqref{muc} and update $\mu_k$ from \eqref{miuk}.
            \State Compute the decent ratio ${\rho _k}$ by \eqref{rhok1}.
              \If{${\rho _k} < {\eta _1}$}
            \State ${\beta _{k + 1}} = {\beta _k}$.
                \While{${\beta _{k + 1}} > {\gamma _1}{\beta _k}$}
            \State ${\beta _{k + 1}} = \left\| {{u_{j + 1}}} \right\|/{\lambda _{j + 1}}$.
            \State Set $j = j + 1$.
                \EndWhile
              \Else
            \State Set flag = 1.
            \State ${x_{k + 1}} = {x_k} + {d_k}$.
            \State Update ${B_{k + 1}}$.
                \If{${\rho _k} > {\eta _2}$}
            \State ${\beta _{k + 1}} = {\gamma _2}{\beta _k}$.
                \Else
            \State ${\beta _{k + 1}} = {\beta _k}$.
                \EndIf
              \EndIf
        \EndWhile
            \State Set $k = k + 1$.
        \EndWhile
\end{algorithmic}
\end{algorithm}

In $\mathrm{SSARC_qK}$, we utilize Algorithm \ref{Alg1} to solve the systems at line 4. At line 7, we let $i^+$ be the index of the smallest shift for which no negative curvature is detected. Consequently, no negative curvature is detected for any $i^+\leq i \leq m$, which also implies that $B^Z_k+\lambda_iI \succ 0$ for any $i^+\leq i\leq m$. If $u_*\approx u(\lambda_i)$ denotes the solution identified by Algorithm \ref{Alg1} corresponding to $\lambda_i$, we retain the one that most closely satisfies $\beta_k\lambda_i=\|u_*\|$ at line 8. These steps ensure that the solution $u$ we obtain is an approximate global minimizer of \eqref{tk}. 

When an iteration is unsuccessful, Algorithm \ref{Alg2} does not need to recalculate $u$. Instead, it simply shifts its attention to the next shift value and the associated search direction, which has already been computed. For a given $\lambda$, an exact solution $u(\lambda)$ would satisfy condition \eqref{7c}. If $\lambda_k$ is the shift used at iteration $k$ and the iteration is unsuccessful, we search for the smallest $\lambda_j > \lambda_k$, i.e., the smallest $j>k$ such that $\beta(\lambda_j):= \|u_j\|/\lambda_j \leq \gamma_1\beta_k$. If this search were to fail, it would indicate that our sample of shifts does not contain sufficiently large values. 

\textcolor{blue}{Most of the operations in the above two paragraphs are the same as Algorithm 3 in \cite{Dussault19} since $\mathrm{SSARC_qK}$ is the extension of Algorithm 3 in \cite{Dussault19} to constrained optimization. Moreover, for clarity and ease of understanding, we have made the descriptions of $\mathrm{SSARC_qK}$ in this paper and Algorithm 3 in \cite{Dussault19} consistent. The main differences between them lie in the handling of constraints and the convergence analysis. This is specifically reflected in lines 3, 5, 10, 11, 12, 22 of $\mathrm{SSARC_qK}$ and in the theoretical analysis in Section 3.2 and Section 3.3.}
\section{Global convergence}\label{sec3}
In this section, we shall prove that the iterates generated by Algorithm \ref{Alg2} ($\mathrm{SSARC_{q}K}$) are globally convergent to a first-order critical point for the problem \eqref{NLP}. First of all, we state some assumptions necessary for the global convergence analysis. Some preliminary results are given in Subsection \ref{Preliminary results}. The feasibility and the optimality are shown in the remaining two subsections.

Let $\{x_k\}$ be the infinite sequence generated by $\mathrm{SSARC_{q}K}$, where we suppose that the algorithm does not stop at a KKT point. Now, we make some common assumptions as follows.\\
\\
  $\bf(A1)$~~The function $f(x)$ is uniformly bounded from below. The functions $g(x)$ and $c(x)$ are
  uniformly bounded.\\ 
  \\
  $\bf(A2)$~~$g(x)$, $c(x)$ and $\nabla c_i(x)$ are Lipschitz continuous in some open set $\mathcal{C}\supseteq \{x_k\}$.\\
  \\
 $\bf(A3)$~~There exists a constant $\kappa_{\rm B}\geq0$ so that $ \|B_k\| \leq \kappa_{\rm B}$ for all $k$.\\
 \\
 $\bf(A4)$~~There exists a constant $\kappa_{\rm n}>0$ so that for all $k$,
  \begin{eqnarray}\label{aa1g1}
   \|v_k^c\| \leq \kappa_{\rm n}\|c(x_k)\|. \label{nkc}
   \end{eqnarray}
\subsection{Preliminary results}\label{Preliminary results}
We have the following result about $\Delta q_k^N$ after the vertical step $v_k$ is calculated by \eqref{nk}.\\
\begin{lemma}\label{lem3.1}
\textcolor{blue}{Let Assumption $\mathrm{(A4)}$ hold.} 
Then
\begin{eqnarray}\label{lem1}
  {\rm min}\left\{\|c(x_k)\|, \frac{\theta\sqrt{\beta _k}}{\kappa_{\rm n}}\right\}\leq \Delta q_k^N \leq  \|c(x_k)\|
\end{eqnarray}
for all $k\geq0$.
\end{lemma}

\begin{proof}
According to the definition of $q^N(x,v)$ in \eqref{qNdef}, the definition of $\Delta q_k^N$ in \eqref{deltaqNdef}, and \eqref{nk}, we have that
\begin{eqnarray}\label{l11}
\begin{aligned}
  \Delta q_k^N &= \|c(x_k)\|-\|c(x_k)+J_kv_k\| \\
  &= \|c(x_k)\|-\|c(x_k)+\alpha_k(J_kv_k^c+c(x_k))-\alpha_kc(x_k)\| \\
  &= \|c(x_k)\|-(1-\alpha_k)\|c(x_k)\| \\
  &= \alpha_k\|c(x_k)\|.
  \end{aligned}
\end{eqnarray}
We obtain that $0<\alpha_k \leq1$ from \eqref{aa1h}. Thus we get that
\begin{eqnarray*}
  \Delta q_k^N=\alpha_k\|c(x_k)\|\leq \|c(x_k)\|.
\end{eqnarray*}
The right side of the inequality in \eqref{lem1} holds.

From \eqref{aa1h} and \eqref{l11}, we get that
\begin{eqnarray*}
  \Delta q_k^N=\alpha_k\|c(x_k)\|\geq \mbox{min}\left\{\|c(x_k)\|, \frac{\theta\|c(x_k)\|\sqrt{\beta _k}}{\|v_k^c\|}\right\}.
\end{eqnarray*}
Assumption $\mathrm{(A4)}$ implies that
$
  \frac{\|c(x_k)\|}{\|v_k^c\|}\geq \frac{1}{\kappa_{\rm n}}.
$
Thus, we can get that
\begin{eqnarray*}
  \Delta q_k^N \geq {\rm min}\left\{\|c(x_k)\|, \frac{\theta\sqrt{\beta _k}}{\kappa_{\rm n}}\right\}.
\end{eqnarray*}
The left side of the inequality in \eqref{lem1} holds. Then, we complete the proof of the lemma. 
\end{proof}

The following result shows that the vertical step is bounded because the total reduction of the vertical model is bounded.

\begin{lemma}\label{lem3.7}
\textcolor{blue}{Let Assumption $\mathrm{(A4)}$ hold.} Then there exists a constant $\kappa_{\rm mn}>0$ such that
\begin{eqnarray}\label{kbon}
 \kappa_{\rm mn}\|v_k\|\leq\Delta q^N_k \leq\|c(x_k)\|
\end{eqnarray}
for all $k\geq0$.
\end{lemma}
\begin{proof}
We shall prove this in two cases for \eqref{kbon}.

Consider the first possibility that $c(x_k)=0$. We can obtain that $v_k=0$ from \eqref{nk} and Assumption $\mathrm{(A4)}$ implying \eqref{nkc}. Then, the conclusion \eqref{kbon} is obviously true.

Consider the case $c(x_k)\neq0$. Lemma \ref{lem3.1} gives that
\begin{eqnarray}\label{mkN}
 \Delta q^N_k \geq {\rm min}\left\{\|c(x_k)\|, \frac{\theta\sqrt{\beta _k}}{\kappa_{\rm n}}\right\}.
\end{eqnarray}
If $\beta _k \leq \left(\frac{\kappa_{\rm n}\|c(x_k)\|}{\theta}\right)^{2}$, 
it follows from \eqref{mkN} that
$
 \Delta q^N_k \geq \frac{\theta\sqrt{\beta _k}}{\kappa_{\rm n}}.
$

From \eqref{nk} and \eqref{aa1h}, we have that
$
\|v_k\| \leq \sqrt{\beta _k}.
$
Therefore, we can get that
\begin{eqnarray}\label{sigmamkn1}
 \Delta q^N_k \geq \frac{\theta}{\kappa_{\rm n}}\|v_k\| = \kappa_{\rm mn}\|v_k\|
\end{eqnarray}
with $\kappa_{\rm mn}:=\frac{\theta}{\kappa_{\rm n}}$.

If
$
 \beta _k > \left(\frac{\kappa_{\rm n}\|c(x_k)\|}{\theta}\right)^{2},
$
we have from \eqref{mkN} that
\begin{eqnarray}\label{mkN2}
  \Delta q^N_k \geq\|c(x_k)\|.
\end{eqnarray}
Since $0<\alpha_k\leq1$ from \eqref{aa1h},  \eqref{nk} and \eqref{nkc}  yield
\begin{eqnarray}\label{star}
 \|v_k\|=\alpha_k\|v^c_k\|\leq\|v^c_k\|\leq\kappa_{\rm n}\|c(x_k)\|.
\end{eqnarray}
\eqref{star} and \eqref{mkN2} imply that
\begin{eqnarray}\label{sigmamkn2}
  \Delta q^N_k \geq\frac{1}{\kappa_{\rm n}}\|v_k\|\geq\frac{\theta}{\kappa_{\rm n}}\|v_k\|=\kappa_{\rm mn}\|v_k\|
\end{eqnarray}
since $0<\theta\leq1$. Then, both \eqref{sigmamkn1} and \eqref{sigmamkn2} hold, and we already get that $\Delta q^N_k\leq\|c(x_k)\|$ from \eqref{lem1}. Therefore, the lemma is proved.
\end{proof}

We now turn to the horizontal step $h_k$ computed by \eqref{tn} and satisfying \eqref{rk}. 
The following lemma presents a lower bound of the decrease brought by horizontal step $h_k$. We need a stronger condition on $\beta_k$ as follows. 
\begin{equation}\label{bdbetak}
\beta_k\leq \beta_{\rm max}, k\geq 0
\end{equation}
for some constant $\beta_{\rm max}>0$. Note that imposing a bound on the regularization is standard practice for adaptive regularization methods \cite{Cartis12,Birgin2017MP}.
\begin{lemma}\label{lem3.2}
\textcolor{blue}{Let Assumption $\mathrm{(A3)}$ hold} and $r_k^Tu_k=0$. 
Then there exists a constant $\kappa_{\rm qkh}\geq 0$ such that
\begin{eqnarray}\label{lema2}
  \Delta q^H_k\geq \kappa_{\rm qkh}\sqrt{\beta_k\|g_k^Z\|^3}
\end{eqnarray}
for all $k\geq 0$.
\end{lemma}

\begin{proof}
\textcolor{blue}{According to the definition of $q^H(x,B,h)$ in \eqref{qHdef}, the definition of $\Delta q^H_k$ in \eqref{deltaqHdef}, the definition of $g_k^Z$, $B_k^Z$ in \eqref{Bzkgzkdef}, and the definition of $u_k$, $Z_k$ in \eqref{tn},} we can get that
\begin{eqnarray*}
  \Delta q^H_k &=& q^H(x_k,B_k,0)-q^H(x_k,B_k,h_k)\\
  &=& -(g(x_k)+B_kv_k)^Th_k-\frac{1}{2}h_k^TB_kh_k \\
  &=& -u_k^T{g_k^Z}-\frac{1}{2}u_k^TB_k^Zu_k.
\end{eqnarray*}
Multiplying \eqref{8a} by $u_k$ and using the assumption that $r_k^Tu_k=0$, we obtain that
\begin{eqnarray*}
-u_k^T{g_k^Z} = u_k^T(B_k^Z+\lambda_kI)u_k.
\end{eqnarray*}
Hence,
\begin{eqnarray*}
  \Delta q^H_k &=& u_k^T(B_k^Z+\lambda_kI)u_k-\frac{1}{2}u_k^TB_k^Zu_k\\
  &=& \frac{1}{2}u_k^TB_k^Zu_k+\lambda_k\|u_k\|^2.
\end{eqnarray*}
According to the selection of the \textcolor{blue}{shift $\lambda_k$ in \eqref{shiftbd},} we have that
\begin{eqnarray}\label{k00}
\lambda_k\geq \lambda_0
\end{eqnarray}
for all $k$.
Thus, using \eqref{8b}, we obtain that
\begin{eqnarray}\label{qkh1}
 \textcolor{blue}{\Delta q_k^H} &\textcolor{blue}{\geq}& \textcolor{blue}{-\frac{1}{2}\lambda_k\|u_k\|^2+\lambda_k\|u_k\|^2} \nonumber\\
 &\textcolor{blue}{\geq}& \textcolor{blue}{\frac{\|u_k\|^3}{2t\beta_k}}
 \textcolor{blue}{=} \textcolor{blue}{\|u_k\|^{\frac{3}{2}}\frac{\|u_k\|^{\frac{3}{2}}}{2t\beta_k}}\textcolor{blue}{\underset{\geq}{\eqref{8c},\eqref{k00}} } \textcolor{blue}{\frac{\lambda_0^{\frac{3}{2}}\sqrt{\beta_k}}{2t^{\frac{5}{2}}}\|u_k\|^{\frac{3}{2}}}.
\end{eqnarray}

We have from \eqref{8a} that
  \begin{equation}\label{rkgkZ}
  \|r_k-g_k^Z\| = \|(B_k^Z+\lambda_kI)u_k\| \leq \|B_k^Z+\lambda_kI\|\|u_k\|.
  \end{equation}
\textcolor{blue}{From the definition of $Z_k$ in \eqref{tn}, $\|Z_k\|$ is bounded, that is, there exists a constant} \textcolor{blue}{$\kappa_{\rm Z}>0$ so that}
\begin{equation}\label{Zkbd}
\textcolor{blue}{\|Z_k\|\leq \kappa_{\rm Z}.}
\end{equation}
\textcolor{blue}{Then Assumption $\mathrm{(A3)}$, \eqref{Bzkgzkdef}, \eqref{shiftbd} and \eqref{Zkbd} imply that}
\begin{eqnarray*}
\textcolor{blue}{ \|B_k^Z+\lambda_kI\|\leq\|Z_k^T\|\|B_k\|\|Z_k\|+\lambda_k \leq \kappa_{\rm B}\kappa_{\rm Z}^2+ 10^{e_U},}
\end{eqnarray*}
\textcolor{blue}{that is,} $$\textcolor{blue}{\|B_k^Z+\lambda_kI\|\leq M}$$ \textcolor{blue}{with} \textcolor{blue}{$M:=\kappa_{\rm B}\kappa_{\rm Z}^2+ 10^{e_U}$.}
\textcolor{blue}{So, from this inequality and \eqref{rkgkZ} we can obtain that}
\begin{equation}\label{ugeqM}
 \textcolor{blue}{ \|r_k-g_k^Z\|\leq M\|u_k\|.}
\end{equation}
\textcolor{blue}{From \eqref{rk}, we have that} $$\textcolor{blue}{\|r_k-g_k^Z\|\geq \|g_k^Z\|-\|r_k\| \geq \|g_k^Z\|-\xi \|u_k\|^{1+\zeta}.}$$ 
\textcolor{blue}{Then it follows from \eqref{ugeqM} that}
$$\textcolor{blue}{\xi \|u_k\|^{1+\zeta}\geq  \|g_k^Z\|-\|r_k-g_k^Z\|\geq \|g_k^Z\|-M\|u_k\|.}$$
\textcolor{blue}{Then we can obtain that}
$$\textcolor{blue}{\|u_k\|\geq \frac{\|g_k^Z\|}{M+\xi \|u_k\|^{\zeta}}.}$$
\textcolor{blue}{\eqref{8c}, \eqref{shiftbd} and \eqref{bdbetak} yield that} $$\textcolor{blue}{\|u_k\|\leq t\lambda_k\beta_k\leq t\beta_{\rm max}10^{eU}.}$$
\textcolor{blue}{Combining the two above inequalities, we have that}
\begin{equation}\label{uk3}
\textcolor{blue}{\|u_k\| \geq \frac{\|g_k^Z\|}{M+\xi (t\beta_{\rm max}10^{eU})^{\zeta}}.}
\end{equation}
Substitute \eqref{uk3} into \eqref{qkh1}, we can get that
\begin{eqnarray*}
  \Delta q^H_k\geq \kappa_{\rm qkh}\sqrt{\beta_k\|g_k^Z\|^3}
\end{eqnarray*}
with $\kappa_{\rm qkh}:=\frac{\lambda_0^{\frac{3}{2}}}{2t^{\frac{5}{2}}(M+\xi (t\beta_{\rm max}10^{eU})^{\zeta})^{\frac{3}{2}}}$. 
Then the lemma is proved.
\end{proof}

The following lemma gives a useful bound for the horizontal step $h_k$.
\begin{lemma}\label{lem3.3}
\textcolor{blue}{Let Assumption $\mathrm{(A3)}$ hold.} 
Then the horizontal step
\begin{eqnarray}\label{lema3}
  \|h_k\|\leq 3\beta_k {\rm max}\left\{t{\rm \kappa_B}, \sqrt{\frac{t\|g_k^Z\|}{\beta_k}}\right\}
\end{eqnarray}
for all $k\geq0$.
\end{lemma}
\begin{proof}
Let
\begin{eqnarray}\label{mH}
 m^H(x_k,B_k,h):=(g(x_k)+B_kv_k)^Th+\frac{1}{2}h^TB_kh+\frac{1}{3t}\beta_k^{-1}\|h\|^3.
\end{eqnarray}
\textcolor{blue}{By the definition of $u_k, Z_k$ in \eqref{tn},} substituting $h_k$ into \eqref{mH}, we get from \eqref{qkh1} that
\begin{eqnarray}\label{mHk}
\begin{aligned}
 \Delta m^H_k & := m^H(x_k,B_k,0)-m^H(x_k,B_k,h_k)\\
 &=-(g_k+B_kv_k)^Th_k-\frac{1}{2}h^T_kB_kh_k-\frac{1}{3t\beta_k}\|h_k\|^3 \\
 &=-{g_k^Z}^Tu_k-\frac{1}{2}u_k^TB_k^Zu_k-\frac{1}{3t\beta_k}\|u_k\|^3 \\
 &=\Delta q_k^H-\frac{1}{3t\beta_k}\|u_k\|^3 \geq \frac{\|u_k\|^3}{6t\beta_k} > 0.
 \end{aligned}
\end{eqnarray}
Now, consider
\begin{eqnarray*}
 -\Delta m_k^H &=& m^H(x_k,B_k,h_k)-m^H(x_k,B_k,0) \\
 &=&(g_k+B_kv_k)^Th_k+\frac{1}{2}h^T_kB_kh_k+\frac{1}{3t\beta_k}\|h_k\|^3 \\
 &\geq& -\|g_k^Z\|\|h_k\|-\frac{1}{2}\|B_k\|\|h_k\|^2+\frac{1}{3t\beta_k}\|h_k\|^3 \\
 &=& \left(\frac{1}{9t\beta_k}\|h_k\|^3-\|g_k^Z\|\|h_k\|\right)+\left(\frac{2}{9t\beta_k}\|h_k\|^3-\frac{1}{2}\|B_k\|\|h_k\|^2\right).
\end{eqnarray*}
When $\|h_k\|>3\sqrt{t\beta_k\|g_k^Z\|}\|$, we have that
\begin{eqnarray*}
 \frac{1}{9t\beta_k}\|h_k\|^3-\|g_k^Z\|\|h_k\|>0.
\end{eqnarray*}
When $\|h_k\|>\frac{9t\beta_k}{4}\|B_k\|$, we have that
\begin{eqnarray*}
 \frac{2}{9t\beta_k}\|h_k\|^3-\frac{1}{2}\|B_k\|\|h_k\|^2>0.
\end{eqnarray*}
So, when $\|h_k\|>3 \beta_k \max\left\{t\|B_k\|, \sqrt{\frac{t\|g_k^Z\|}{\beta_k}}\right\}$, we get that
$
 -\Delta m_k^H>0.
$
But this is contradictory to \eqref{mHk}. Then, we have from Assumption $\mathrm{(A3)}$ that
$$ \|h_k\|\leq 3 \beta_k \max\left\{t\|B_k\|, \sqrt{\frac{t\|g_k^Z\|}{\beta_k}}\right\} \leq 3 \beta_k \max\left\{t{\rm \kappa_B}, \sqrt{\frac{t\|g_k^Z\|}{\beta_k}}\right\}.$$
The conclusion of the lemma holds.
\end{proof}

The following lemma gives a useful bound on the overall step $d_k$ from the result of Lemma \ref{lem3.3}.
\begin{lemma}\label{lem3.4}
\textcolor{blue}{Let Assumptions $\mathrm{(A1)}$, $\mathrm{(A3)}$, and $\mathrm{(A4)}$ hold.} Let $\mathcal{I}$ be an infinite index such that
\begin{eqnarray}\label{37}
 \|g^Z_k\|\geq\epsilon,~{\rm for~all}~k\in \mathcal{I}~{\rm and~some}~\epsilon>0,
\end{eqnarray}
and
\begin{eqnarray}\label{38}
 \sqrt{\frac{\beta_k\|g^Z_k\|}{t}}\rightarrow0,~{\rm as}~k\rightarrow\infty,~k\in\mathcal{I}.
\end{eqnarray}
Then, there exists a constant $M_{\rm d}>0$ such that
\begin{eqnarray}\label{d39}
 \|d_k\|\leq  M_{\rm d}\sqrt{\beta_k},
\end{eqnarray}
\textcolor{blue}{for sufficiently large $k\in\mathcal{I}$.}
\end{lemma}
\begin{proof}
According to \eqref{37}, we have that
$$
 \sqrt{\frac{t\|g^Z_k\|}{\beta_k}} =  \|g_k^Z\|\sqrt{\frac{t}{\beta_k\|g^Z_k\|}}\geq \epsilon \sqrt{\frac{t}{\beta_k\|g^Z_k\|}}.
$$
Then, we have from \eqref{38} that
\begin{eqnarray*}
\epsilon \sqrt{\frac{t}{\beta_k\|g^Z_k\|}}\rightarrow\infty,~{\rm as}~k\rightarrow\infty,~k\in\mathcal{I}.
\end{eqnarray*}
So, we get from \eqref{lema3} that
\begin{eqnarray}\label{hk40}
 \|h_k\|\leq 3\sqrt{t\beta_k\|g_k^Z\|}.
\end{eqnarray}

\textcolor{blue}{By Assumptions $\mathrm{(A1)}$, $\mathrm{(A3)}$, $\mathrm{(A4)}$ and \eqref{kbon}, we get that there exists a constant $M_{\rm g}$ so that for all k,}
\begin{eqnarray}\label{gkzbd}
\textcolor{blue}{  \|g_k^Z\|=\|Z^T_k(g_k+B_kv_k)\|\leq M_{\rm g}.}
\end{eqnarray}
\textcolor{blue}{From \eqref{nk} and \eqref{aa1h}, we have that
$\|v_k\| \leq \sqrt{\beta _k}.$ 
Using this bound of $\|v_k\|$, the definition of $d_k$ in \eqref{dkdef}, \eqref{hk40}, and \eqref{gkzbd}, we obtain that}
\begin{eqnarray*}
\textcolor{blue}{\|d_k\|}&\textcolor{blue}{\leq}& \textcolor{blue}{\|v_k\|+\|h_k\|}\\
& \textcolor{blue}{\leq} &\textcolor{blue}{\sqrt{\beta _k}+3\sqrt{t\beta_k\|g_k^Z\|}}\\
& \textcolor{blue}{\leq} &\textcolor{blue}{(1+3\sqrt{t M_{\rm g}})\sqrt{\beta_k}}.
\end{eqnarray*}
Let \textcolor{blue}{${M_{\rm d}}:=1+3\sqrt{tM_{\rm g}}.$} Then, we have that $\|d_k\|\leq M_{\rm d}\sqrt{\beta_k},$ and the lemma is proved.
\end{proof}

We next investigate the error between the penalty function and its overall model at the new iterate $x_k+d_k$.
\begin{lemma}\label{lem3.5}
\textcolor{blue}{Let Assumptions $\mathrm{(A2)-(A4)}$ hold.} Then, \textcolor{blue}{for sufficiently large $k\in\mathcal{I}$,} the step generated by the algorithm satisfies the bounds
\begin{eqnarray}
\phi(x_k+d_k,\mu_k)-q(x_k,B_k,\mu_k,d_k)&\leq& \kappa_0(1+\mu_k)\|d_k\|^2 \label{pmerror} \\
  &\leq&\kappa_{\rm m}(1+\mu_k)\beta_k \label{fimerror}
\end{eqnarray}
for some constants $\kappa_0$ and $\kappa_{\rm m}$.
\end{lemma}
\begin{proof}
From Assumption $\mathrm{(A2)}$ and the Taylor expansion of $f(x_k+d_k)$,
\begin{eqnarray}\label{39}
  f(x_k+d_k)-f(x_k)-g^T_kd_k \leq \kappa_{\rm fx}\|d_k\|^2
\end{eqnarray}
for some Lipschitz constant $\kappa_{\rm fx}$. Similarly, the Lipschitz continuity of $J(x)$ implies that
\begin{eqnarray}\label{40}
\begin{aligned}
 \|c(x_k+d_k)\|-\|c(x_k)+J_kd_k\|&\leq\|c(x_k+d_k)-J_kd_k-c(x_k)\| \\
 &\leq \kappa_{\rm cx}\|d_k\|^2
 \end{aligned}
\end{eqnarray}
for another Lipschitz constant $\kappa_{\rm cx}$. Assumption $\mathrm{(A3)}$ implies that
\begin{eqnarray}\label{41}
  d^T_kB_kd_k \leq \kappa_{\rm B}\|d_k\|^2.
\end{eqnarray}
Hence, from \eqref{39}, \eqref{40}, \eqref{41}, the definition of $q(x,B,\mu,d)$ in \eqref{phimodelqdef} and the definition of $\phi(x,\mu)$ in \eqref{phidef}, we obtain that
\begin{eqnarray*}
&&\phi(x_k+d_k,\mu_k)-q(x_k,B_k,\mu_k,d_k)\\
&=& f(x_k+d_k)-f(x_k)-g^T_kd_k-\frac{1}{2}d^T_kB_kd_k+\mu_k\|c(x_k+d_k)\|\\
&-&\mu_k\|c(x_k)+J_kd_k\|\\
&\leq& (\kappa_{\rm fx}+\frac{1}{2}\kappa_{\rm B}+\mu_k\kappa_{\rm cx})\|d_k\|^2\\
&\leq& \kappa_0(1+\mu_k)\|d_k\|^2
\end{eqnarray*}
with $\kappa_0 :={\rm max}\{\kappa_{\rm fx}+\frac{1}{2}\kappa_{\rm B}, \kappa_{\rm cx}\}$. Then \eqref{pmerror} holds.

From \eqref{d39}, we get that
$$\phi(x_k+d_k,\mu_k)-q(x_k,B_k,\mu_k,d_k) \leq  \kappa_0(1+\mu_k){M_{\rm d}}^2\textcolor{blue}{\beta_k}= \kappa_{\rm m}(1+\mu_k)\beta_k
$$
with $\kappa_{\rm m}:=\kappa_0 M_{\rm d}^2$. Then \eqref{fimerror} holds.

Therefore, we complete the proof of the lemma.
\end{proof}

For convenience, define the auxiliary function
\begin{equation}\label{Phidef}
\Phi(x,\mu):=\frac{\phi(x,\mu)-f_{\rm min}}{\mu}.
\end{equation}

Below we provide some important results for the auxiliary function $\Phi(x,\mu)$.
\begin{lemma}\label{lem3.6}
\textcolor{blue}{Let Assumptions ${\rm (A1)}$ and ${\rm (A4)}$ hold.} Let $\{k_i\}(i\geq0)$ be the index set of the successful iterations from the algorithm. Then
\begin{eqnarray}\label{varphi1}
  \Phi(x_{k_i},\mu_{k_i}) \geq 0
\end{eqnarray}
and for any $j>0$,
\begin{eqnarray}\label{varphi2}
  \Phi(x_{k_{i+j}},\mu_{k_{i+j}})\leq\Phi(x_{k_i},\mu_{k_i})-\eta_1\frac{\Delta q_{k_i}}{\mu_{k_i}}.
\end{eqnarray}
\end{lemma}
\begin{proof}
Assumption $\mathrm{(A1)}$ implies that
\begin{eqnarray}\label{fmin}
 f(x_k)\geq f_{\rm min}
\end{eqnarray}
for all $k\geq0$ and some constant $f_{\rm min}$.

Due to the fact that $\mu_k>0$ for all $k\geq0$, $\phi(x_k,\mu_k)\geq f_{\rm min}$ from \eqref{fmin}, and the definition of $\phi(x,\mu)$ in \eqref{phidef}, we have that
\begin{eqnarray*}
\frac {\phi(x_k,\mu_k)-f_{\rm min}}{\mu_k} \geq 0.
\end{eqnarray*}
Thus, \eqref{varphi1} holds.

Next, we prove the second conclusion using a mathematical induction. 

\textcolor{blue}{Since $k_i$ and $k_{i+1}$ are consecutive successful iterations, we have that}
\textcolor{blue}{$$x_l=x_{k_{i+1}}$$}
\textcolor{blue}{for $k_i+1\leq l \leq k_{i}.$}
Then, it follows from $\rho_{k_i}\geq \eta_1$ that
\begin{eqnarray*}
  \phi(x_{k_{i+1}},\mu_{k_i})-f_{\rm min}&=&\phi(x_{k_i+1},\mu_{k_i})-f_{\rm min}\\
  &\leq& \phi(x_{k_i},\mu_{k_i})-f_{\rm min}-\eta_1\Delta q_{k_i}.
\end{eqnarray*}

\textcolor{blue}{Thus, we have from the definition of $\Phi(x,\mu)$ in \eqref{Phidef} that}
\begin{equation}\label{monoPhi}
  \Phi(x_{k_{i+1}},\mu_{k_i})\leq\Phi(x_{k_{i}},\mu_{k_i})-\eta_1\frac{\Delta q_{k_i}}{\mu_{k_i}}.
\end{equation}

\textcolor{blue}{It follows from the definition of $\phi(x,\mu)$ in \eqref{phidef} and the definition of $\Phi(x,\mu)$ in \eqref{Phidef} that}
\begin{eqnarray*}
 &&\textcolor{blue}{\Phi(x_{k_{i+1}},\mu_{k_i})-\Phi(x_{k_{i+1}},\mu_{k_{i+1}})}\\
 &\textcolor{blue}{=}&\textcolor{blue}{\frac {\phi(x_{k_{i+1}},\mu_{k_{i}})-f_{\rm min}}{\mu_{k_{i}}}-\frac {\phi(x_{k_{i+1}},\mu_{k_{i+1}})-f_{\rm min}}{\mu_{k_{i+1}}}}\\
 &\textcolor{blue}{=}&\textcolor{blue}{\frac {f(x_{k_{i+1}})+\mu_{k_{i}}\|c(x_{k_{i+1}})\|-f_{\rm min}}{\mu_{k_{i}}}-\frac {f(x_{k_{i+1}})+\mu_{k_{i+1}}\|c(x_{k_{i+1}})\|-f_{\rm min}}{\mu_{k_{i+1}}}}\\
 &\textcolor{blue}{=}&\textcolor{blue}{(f(x_{k_{i+1}})-f_{\rm min}) \left(\frac{1}{\mu_{k_i}}-\frac{1}{\mu_{k_{i+1}}}\right)\geq 0.}
\end{eqnarray*}

Then the fact that $\mu_{k_i}\leq\mu_{k_{i+1}}$  and \eqref{fmin} give that
\begin{eqnarray*}
 \Phi(x_{k_{i+1}},\mu_{k_i})-\Phi(x_{k_{i+1}},\mu_{k_{i+1}})=(f(x_{k_{i+1}})-f_{\rm min})
 \left(\frac{1}{\mu_{k_i}}-\frac{1}{\mu_{k_{i+1}}}\right)\geq 0,
\end{eqnarray*}
which implies that $$\Phi(x_{k_{i+1}},\mu_{k_{i+1}})\leq\Phi(x_{k_{i+1}},\mu_{k_i}).$$
Combining this inequality and \eqref{monoPhi}, we have that \eqref{varphi2} holds for $j=1$.

\textcolor{blue}{Furthermore, since Assumption ${\rm (A4)}$ holds, from \eqref{deltam} and Lemma \ref{lem3.1}, we have that}
\begin{eqnarray*}
\textcolor{blue}{\Delta q_{k_i}\geq\nu\mu_{k_i}\Delta q^N_{k_i}\geq {\rm min}\left\{\|c(x_{k_i})\|, \frac{\theta\sqrt{\beta_{k_i}}}{\kappa_{\rm n}}\right\}\geq 0.}
\end{eqnarray*}
\textcolor{blue}{It then follows from \eqref{varphi2} with $j=1$ that
 $\Phi(x_{k_{i+1}},\mu_{k_{i+1}})\leq\Phi(x_{k_i},\mu_{k_i})$, that is, $\{\Phi(x_{k_i},\mu_{k_i})\}$ is monotonic. }

We now assume that \eqref{varphi2} holds for $j-1$. In the following, we show that \eqref{varphi2} holds for $j$.

Due to
\begin{eqnarray*}
 \Phi(x_{k_{i+j-1}},\mu_{k_{i+j-1}})\leq\Phi(x_{k_i},\mu_{k_i})-\eta_1\frac{\Delta q_{k_i}}{\mu_{k_i}},
\end{eqnarray*}
we have that
\begin{eqnarray*}
 \Phi(x_{k_{i+j}},\mu_{k_{i+j}})\leq\Phi(x_{k_{i+j-1}},\mu_{k_{i+j-1}})\leq\Phi(x_{k_i},\mu_{k_i})-\eta_1\frac{\Delta q_{k_i}}{\mu_{k_i}}.
\end{eqnarray*}
Therefore, \eqref{varphi2} holds for any $j>0$. Then the proof of this lemma is completed.
\end{proof}
\subsection{Feasibility}
The following theorem shows the feasibility of the algorithm. The index set of successful iterations of the algorithm is denoted as
\begin{eqnarray*}
 \mathcal{S}:=\{k\geq0:~k~{\rm successful ~or ~very ~successful, i.e., \rho_k\geq\eta_1}\}.
\end{eqnarray*}
\begin{theorem}\label{thm3.1}
Let Assumptions $\mathrm{(A1)-(A4)}$ hold. Then
\begin{eqnarray*}
 \lim_{k\rightarrow\infty}c(x_k)=0.
\end{eqnarray*}
\end{theorem}
\begin{proof}
Since the result is trivial if $c(x_l)=0$ for all sufficiently large $l$, we consider an arbitrary infeasible iteration $x_l$ such that $\|c(x_l)\|>0$.
Our first aim is to show that Algorithm \ref{Alg2} ($\mathrm{SSARC_{q}K}$) cannot stop at such a point.

The Lipschitz continuity of $\|c(x)\|$ from Assumption (A2) gives that
\begin{eqnarray}\label{L1}
 | \|c(x)\|-\|c(x_l)\| | \leq L_1\|x-x_l\|
\end{eqnarray}
for some $L_1>0$ and all $x$. Thus, \eqref{L1} certainly holds for all $x$ in an open ball
\begin{eqnarray}\label{Ol}
 \mathcal{O}_l\overset{\rm def}{=}\left\{x~|~\|x-x_l\|<\frac{\|c(x_l)\|}{2L_1}\right\}.
\end{eqnarray}
It then follows immediately from \eqref{L1} and \eqref{Ol} that
\begin{eqnarray*}
| \|c(x)\|-\|c(x_l)\| |<\frac{1}{2}\|c(x_l)\|.
\end{eqnarray*}
Thus, we have that
\begin{eqnarray}\label{cl}
 \|c(x)\|>\frac{1}{2}\|c(x_l)\|,
\end{eqnarray}
which ensures that $c(x)$ is bounded away from zero for any $x\in\mathcal{O}_l$. 

From Assumption (A4) and Lemma \ref{lem3.1}, we have that \eqref{lem1} holds. Together with \eqref{deltam} and \eqref{cl}, we get that
\begin{equation}\label{cll}
 \Delta q_k  \geq \nu\mu_k{\rm min}\left\{\|c(x_k)\|, \frac{\theta\sqrt{\beta_k}}{\kappa_{\rm n}}\right\}\geq \nu\mu_k{\rm min}\left\{\frac{1}{2}\|c(x_l)\|, \frac{\theta\sqrt{\beta_k}}{\kappa_{\rm n}}\right\}
\end{equation}
for any $x_k\in\mathcal{O}_l$.


Let
\begin{eqnarray}\label{sigmal}
 \beta^0_l={\rm min}\left\{\left(\frac{\kappa_{\rm n}\|c(x_l)\|}{2\theta}\right)^{2}, \left(\frac{(1-\eta_2)\nu\theta\mu_{-1}}{\kappa_{\rm m}\kappa_{\rm n}(1+\mu_{-1})}\right)^{2}\right\}.
\end{eqnarray}

\textcolor{blue}{From Assumptions $\mathrm{(A2)-(A4)}$, \eqref{fimerror} holds by Lemma \ref{lem3.5}.}

If $\beta_k\leq\beta^0_l$, it follows from \eqref{cll} and \eqref{sigmal} that
\begin{eqnarray*}
 \Delta q_k\geq\nu\mu_k\frac{\theta\sqrt{\beta_k}}{\kappa_{\rm n}},
\end{eqnarray*}
which, together with \eqref{fimerror} and the fact that $\{\mu_k\}$ is nondecreasing, gives that
\begin{eqnarray*}
1-\rho_k&=&\frac{\phi(x_k+d_k,\mu_k)-q(x_k,B_k,\mu_k,d_k)}{\Delta q_k}  \\
&\leq&\frac{\kappa_{\rm m}\kappa_{\rm n}(1+\mu_k)\sqrt{\beta_k}}{\nu\theta\mu_k}\leq \frac{\kappa_{\rm m}\kappa_{\rm n}(1+\mu_{-1})\sqrt{\beta_k}}{\nu\theta\mu_{-1}}\leq 1-\eta_2.
\end{eqnarray*}
Thus, so long as $\beta_k\leq\beta^0_l$, a very successful iteration will be obtained for any $x_k\in\mathcal{O}_l$.

Next, we show that there must be a first successful iterate $x_{k_j},k_j>k_0:=l$, not in $\mathcal{O}_l$.
Suppose that all successful iterates $x_{k_i}\in\mathcal{O}_l$. Then
\begin{eqnarray}\label{deltal}
  \beta_{k_i}\geq\beta^0:=\gamma_1\min\{\beta_l,\beta^0_l\}.
\end{eqnarray}

Since iteration $k_i\in\mathcal{S}$ is successful, \eqref{varphi2} holds from Assumption (A1) and Lemma \ref{lem3.6}. Then from \eqref{varphi2},  \eqref{cll} and \eqref{deltal} we can get that
\begin{eqnarray*}
  \Phi(x_{k_{i+1}},\mu_{k_{i+1}})&\leq&\Phi(x_{k_i},\mu_{k_i})-\eta_1\frac{\Delta q_{k_i}}{\mu_{k_i}}\\
  &\leq&\Phi(x_{k_i},\mu_{k_i})-\eta_1\nu\min\left\{\frac{1}{2}\|c(x_l)\|, \frac{\theta\sqrt{\beta^0}}{\kappa_{\rm n}}\right\}.
\end{eqnarray*}
By summing the first $j$ of these successful iterations, we have that
\begin{eqnarray*}
  \Phi(x_{k_{i+j}},\mu_{k_{i+j}})&\leq&\Phi(x_{k_i},\mu_{k_i})-\eta_1\frac{\Delta q_{k_i}}{\mu_{k_i}}-\cdot\cdot\cdot-\eta_1\frac{\Delta q_{k_{i+j-1}}}{\mu_{k_{i+j-1}}}\\
  &\leq&\Phi(x_{k_i},\mu_{k_i})-j\eta_1\nu\min\left\{\frac{1}{2}\|c(x_l)\|, \frac{\theta\sqrt{\beta^0}}{\kappa_{\rm n}}\right\}.
\end{eqnarray*}

Since the right-hand side of this inequality can be made arbitrarily negative by increasing $j$, $\Phi(x_{k_{i+j}},\mu_{k_{i+j}})$ must eventually be negative. However, this \textcolor{blue}{contradicts} \eqref{varphi1}. Our hypothesis that $x_{k_i}$ \textcolor{blue}{($k_i\in\mathcal{S}$, $k_i\geq l$)} remain in $\mathcal{O}_l$ must be false, and thus there must be a first iterate $x_{k_j},j>0$, not in $\mathcal{O}_l$.

We now consider the iterates between $x_{k_j}$ and $x_{k_0}=x_l$. \eqref{varphi2} and \eqref{cll} imply that
\begin{eqnarray}\label{varphik}
  \Phi(x_{k_{m}},\mu_{k_{m}})\leq\Phi(x_{k_i},\mu_{k_i})-\eta_1\nu\min\left\{\frac{1}{2}\|c(x_l)\|, \frac{\theta\sqrt{\beta_{k_i}}}{\kappa_{\rm n}}\right\}
\end{eqnarray}
for any $i < m \leq j$, since $x_{k_i}\in\mathcal{O}_l$. If
\begin{eqnarray*}
  \beta_{k_i}\geq\left(\frac{\kappa_{\rm n}\|c(x_l)\|}{2\theta} \right)^{2},
\end{eqnarray*}
\eqref{varphi2} and \eqref{varphik} with $m=j$ give that
\begin{eqnarray}\label{varphikj1}
  \Phi(x_{k_j},\mu_{k_j})\leq\Phi(x_{k_i},\mu_{k_i})-\frac{1}{2}\eta_1\nu\|c(x_l)\|
  \leq\Phi(x_l,\mu_l)-\frac{1}{2}\eta_1\nu\|c(x_l)\|.
\end{eqnarray}
On the other hand, if
$
  \beta_{k_i}<\left(\frac{\kappa_{\rm n}\|c(x_l)\|}{2\theta} \right)^{2}
$
for all $0\leq i\leq j$, we have that
\begin{eqnarray}\label{varphiki}
 \Phi(x_{k_{i+1}},\mu_{k_{i+1}})\leq\Phi(x_{k_i},\mu_{k_i})-
 \frac{\eta_1\nu\theta\sqrt{\beta_{k_i}}}{\kappa_{\rm n}}
\end{eqnarray}
for each $0\leq i\leq j$. \textcolor{blue}{Next, by taking the sum of \eqref{varphiki} for each $i$ from $0$ to $j-1$,} we get that
\begin{eqnarray}\label{varphi11}
  \Phi(x_{k_j},\mu_{k_j})\leq\Phi(x_l,\mu_l)-\frac{\eta_1\nu\theta}{\kappa_{\rm n}}
  \sum_{i=0}^{j-1}\sqrt{\beta_{k_i}}.
\end{eqnarray}
But as $x_{k_j}\notin\mathcal{O}_l$, \eqref{d39} and \eqref{Ol} imply that
\begin{eqnarray}\label{cxl}
 \frac{\|c(x_l)\|}{2L_1}\leq\|x_{k_j}-x_l\|\leq \sum_{i=0}^{j-1}\|x_{k_{i+1}}-x_{k_i}\| \leq \textcolor{blue}{M_{\rm d}}
 \sum_{i=0}^{j-1}\sqrt{\beta_{k_i}}.
\end{eqnarray}
Then \eqref{varphi11} and \eqref{cxl} give that
\begin{eqnarray}\label{varphikj2}
  \Phi(x_{k_j},\mu_{k_j})\leq\Phi(x_l,\mu_l)-\frac{\eta_1\nu\theta\|c(x_l)\|}{2M_{\rm d}L_1\kappa_{\rm n}}.
\end{eqnarray}
Thus in all cases \eqref{varphikj1} and \eqref{varphikj2} imply that
\begin{eqnarray}\label{Them1}
 \Phi(x_{k_j},\mu_{k_j})\leq\Phi(x_l,\mu_l)-\kappa_c\|c(x_l)\|,
\end{eqnarray}
where $\kappa_{\rm c}
 :=\frac{1}{2}\eta_1\nu\textcolor{blue}{{\min}}\left\{1,\frac{\theta}{M_{\rm d}L_1\kappa_{\rm n}} \right\}$. \eqref{Them1} proves that $\|c(x_l)\|$ cannot be bounded away from zero since $l$ is arbitrary and  $\{\Phi(x_{k_j},\mu_{k_j})\}$ is decreasing and bounded from below from Lemma \ref{lem3.6}.
\end{proof}
\subsection{Optimality}
The feasibility of the algorithm has been proved in Theorem \ref{thm3.1}. Then we turn to the optimality condition.
\begin{lemma}\label{thm3.2}
\textcolor{blue}{Let Assumptions $\mathrm{(A1)}$, $\mathrm{(A3)}$, and $\mathrm{(A4)}$ hold} for all $k$ sufficiently large. Then there exist constants $\mu_{\max}>0$ and $\kappa_{\rm mk}>0$, and an index $k_1$ such that
\begin{eqnarray}\label{mkt1}
 \Delta q_k\geq\Delta q^H_k+(\mu_k-\kappa_{\rm mk})\Delta q^N_k
\end{eqnarray}
and
\begin{eqnarray}\label{mumax}
 \mu_k=\mu_{\max}
\end{eqnarray}
for all $k\geq k_1$.
\end{lemma}
\begin{proof}
We know from \eqref{deltam} that $\mu_k$ satisfies
\begin{equation}\label{mkt2}
  \Delta q_k\geq\nu\mu_k\Delta q^N_k.
\end{equation}
From Assumption (A4) and Lemma \ref{lem3.7}, \eqref{kbon} holds.
Since Assumptions (A1) and (A3) imply that both $g_k$ and $B_k$ are bounded, \eqref{kbon} gives that 
$\|v_k\|\leq\Delta q^N_k/\kappa_{\rm mn}$ and $\Delta q^N_k\leq\|c(x_k)\|$.

From Assumption (A1), both $\|c(x_k)\|$ and $\|g_k\|$ are bounded. Thus, \textcolor{blue}{by the definition of $q^F(x,B,v)$ in $\eqref{qFdef}$ and the definition of $\Delta q^F_k$ in \eqref{deltaqFdef},} we have that
\begin{eqnarray}\label{59}
\begin{aligned}
  \Delta q^F_k &= -g_k^Tv_k-\frac{1}{2}v_k^TB_kv_k  \\
  &\geq-\|g_k\|\|v_k\|-\frac{1}{2}\kappa_{\rm B}\|v_k\|^2 =\|v_k\|(-\|g_k\|-\frac{1}{2}\kappa_{\rm B}\|v_k\|) \\
  &\geq\frac{\Delta q_k^N}{\kappa_{\rm mn}}\left(-\|g_k\|-\frac{1}{2}\kappa_{\rm B}\frac{\|c(x_k)\|}{\kappa_{\rm mn}}\right)\geq-\kappa_{\rm mk}\Delta q_k^N
  \end{aligned}
\end{eqnarray}
for some constant $\kappa_{\rm mk}>0$. Then \eqref{deltam} and \eqref{59} give that
$$
 \Delta q_k =  \Delta q_k^H+\mu_k\Delta q_k^N+\Delta q_k^F \geq  \Delta q_k^H+(\mu_k-\kappa_{\rm mk})\Delta q_k^N.
$$
The conclusion of \eqref{mkt1} is true.

Due to $\Delta q_k^H > 0$, we have that
\begin{eqnarray*}
 \Delta q_k\geq(\mu_k-\kappa_{\rm mk})\Delta q_k^N.
\end{eqnarray*}
Hence \eqref{mkt2} is evidently satisfied for all $\mu_k\geq\frac{\kappa_{\rm mk}}{1-\nu}$. In our algorithm, every increase of an insufficient $\mu_{k-1}$ must be by at least ${\max}\{(\tau_1-1)\mu_0, \tau_2\}$, and therefore the penalty parameter $\mu_k$ can only be increased a finite number of times. The required value $\mu_{\max}$ is at most the first value of $\mu_k$ that exceeds $\frac{\kappa_{\rm mk}}{1-\nu}$. This proves the lemma.
\end{proof}
\begin{lemma}\label{lem3.8}
\textcolor{blue}{Let Assumptions $\mathrm{(A1)}$, $\mathrm{(A3)}$, and $\mathrm{(A4)}$ hold} for all $k$ sufficiently large. Then, there exists a constant $\kappa_{\rm btm}>0$ and an index $k_1$ such that
\begin{eqnarray}\label{kbtm}
 \Delta q_k\geq\kappa_{\rm btm}\Delta q^H_k
\end{eqnarray}
for all $k\geq k_1$.
\end{lemma}
\begin{proof}
Since Assumptions $\mathrm{(A1)}$, $\mathrm{(A3)}$ and $\mathrm{(A4)}$ hold for all $k$ sufficiently large, \eqref{mkt1} and \eqref{mumax} hold from Lemma \ref{thm3.2}.

Consider $k\geq k_1$ as in Lemma \ref{thm3.2}, in which case $\mu_k=\mu_{\max}$.
Let
\begin{eqnarray*}
 \kappa_{\rm btm}:=\frac{1}{2}\min\left\{1,\frac{\nu\mu_{\max}}{\kappa_{\rm mk}-\mu_{\max}}\right\}.
\end{eqnarray*}

On the one hand, suppose that
\begin{eqnarray*}
 -\frac{1}{2}\Delta q^H_k >(\mu_{\max}-\kappa_{\rm mk})\Delta q^N_k.
\end{eqnarray*}
 Since both $\Delta q^H_k$ and $\Delta q^N_k$ are nonnegative, it must be that $\mu_{\max}\leq\kappa_{\rm mk}$, in which case \eqref{deltam} and \eqref{muc} imply that
\begin{eqnarray*}
  \Delta q_k>\frac{\nu\mu_{\max}}{2(\kappa_{\rm mk}-\mu_{\max})}\Delta q^H_k.
\end{eqnarray*}
This proves \eqref{kbtm}.

On the other hand, suppose that
\begin{eqnarray*}
 -\frac{1}{2}\Delta q^H_k\leq(\mu_{\max}-\kappa_{\rm mk})\Delta q^N_k.
\end{eqnarray*}
 In this case, \eqref{mkt1} implies that $\Delta q_k\geq\frac{1}{2}\Delta q^H_k$. Thus,
we complete the proof of the lemma.
\end{proof}
\begin{lemma}\label{sigmabounded}
\textcolor{blue}{Let Assumptions $\mathrm{(A2)-(A4)}$ hold.}
Suppose that the iterate $x_k$ generated by Algorithm \ref{Alg2} is not a first-order critical point for \eqref{NLP}. Then there exists a $\beta_k^0>0$ such that $\rho_k\geq \eta_2$ if $\beta_k\leq \beta_k^0$.
\end{lemma}
\begin{proof}
There are two cases to consider, namely, $c(x_k)\neq 0$ and $c(x_k)= 0$.

Firstly, suppose that $c(x_k)\neq 0$. 

From Assumption $\mathrm{(A4)}$, \eqref{lem1} holds by Lemma \ref{lem3.1}.
Then we have from \eqref{deltam} and \eqref{lem1} that
\begin{equation}\label{case1}
\Delta q_k \geq \nu\mu_k\min\left\{\|c(x_k)\|, \frac{\theta\sqrt{\beta_k}}{\kappa_{\rm n}}\right\}.
\end{equation}
From Assumptions $\mathrm{(A2)-(A4)}$, \eqref{fimerror} holds by Lemma \ref{lem3.5}.
Let
\begin{equation}\label{sigmak01}
\beta_k^0:=\min\left\{\left(\frac{\nu\theta\mu_{-1}(1-\eta_2)}{\kappa_{\rm m}\kappa_{\rm n}(1+\mu_{-1})}\right)^2, \left(\frac{\kappa_{\rm n}\|c(x_k)\|}{\theta}\right)^2\right\}.
\end{equation}
If $\beta_k\leq \beta_k^0$, \eqref{case1} and \eqref{sigmak01} show that
\begin{eqnarray*}
\Delta q_k \geq \nu\mu_k\frac{\theta\sqrt{\beta_k}}{\kappa_{\rm n}},
\end{eqnarray*}
which, together with \eqref{fimerror} and the fact that $\{\mu_k\}$ is nondecreasing, gives that
\begin{eqnarray*}
  1-\rho_k &=& \frac{\phi(x_k+d_k,\mu_k)-q(x_k,B_k,\mu_k,d_k)}{\Delta q_k} \\
   &\leq& \frac{\kappa_{\rm m}\kappa_{\rm n}(1+\mu_k)\sqrt{\beta_k}}{\nu\theta\mu_k}\leq \frac{\kappa_{\rm m}\kappa_{\rm n}(1+\mu_{-1})\sqrt{\beta_k}}{\nu\theta\mu_{-1}} \leq 1-\eta_2,
\end{eqnarray*}
 which implies that $\rho_k\geq \eta_2$ and $\beta_{k+1}\geq \beta_k$.

Now suppose instead that $c(x_k)=0$.  Since \eqref{nk} and \eqref{nkc} imply that $v_k=0$, we have from \eqref{deltam} that $\Delta q_k= \Delta q_k^H$. It then follows from Lemma \ref{lem3.2} that
\begin{equation}\label{nk0modeldecrease}
\Delta q_k \geq  \kappa_{\rm qkh}\sqrt{\beta_k\|g_k^Z\|^3}.
\end{equation}
Let $l_k$ be the index of the last iteration such that $c(x_{l_k})\neq 0$. Then
\begin{equation}\label{mubar}
\mu_k\leq \bar{\mu}:=\max\{\tau_1\mu_{l_k},\mu_{l_k}+\tau_2\}.
\end{equation}
Note that $\|g^Z_k\|\neq 0$ since $x_k$ is not a first-order critical point and $c(x_k)=0$.

Let
\begin{equation}\label{sigmakk0}
\bar{\beta}_k^0 := \left(\frac{\kappa_{\rm qkh}\|g_k^Z\|^{\frac{3}{2}}(1-\eta_2)}{\kappa_{\rm m}(1+ \bar{\mu})}\right)^2.
\end{equation}
If $\beta_k\leq \bar{\beta}_k^0$, we can obtain from \eqref{fimerror}, \eqref{nk0modeldecrease}, \eqref{mubar} and \eqref{sigmakk0} that
\begin{eqnarray*}
1-\rho_k &=& \frac{\phi(x_k+d_k,\mu_k)-q(x_k,B_k,\mu_k,d_k)}{\Delta q_k} \\
&\leq& \frac{\kappa_{\rm m}(1+\mu_k)\sqrt{\beta_k}}{\kappa_{\rm qkh}\|g_k^Z\|^{\frac{3}{2}}}\leq \frac{\kappa_{\rm m}(1+ \bar{\mu})\sqrt{\beta_k}}{\kappa_{\rm qkh}\|g_k^Z\|^{\frac{3}{2}}}\leq 1-\eta_2.
\end{eqnarray*}
Then, we also obtain that $\rho_k\geq \eta_2$ and $\beta_{k+1}\geq \beta_k$, which proves the lemma.
\end{proof}
\begin{theorem}\label{thm3.3}
\textcolor{blue}{Let Assumptions $\mathrm{(A1)-(A4)}$ hold.} 
Then
\begin{eqnarray*}
{\underset {k\rightarrow\infty}{\rm{\liminf}}}~ \|Z^T_kg_k\|=0.
\end{eqnarray*}
\end{theorem}
\begin{proof}
Assume that $Z^T_kg_k$ is bounded away from zero, that is,
\begin{eqnarray}\label{klcg}
  \|Z^T_kg_k\|\geq2\epsilon>0
\end{eqnarray}
for some constant $\epsilon>0$ and for all $k$. 

\textcolor{blue}{$\|B_k\|$ is bounded from Assumption (A3). Due to the definition of $v_k$ in \eqref{nk}, \eqref{aa1h} and Assumption (A4),} $$\textcolor{blue}{\|v_k\|\leq\alpha_k \kappa_{\rm n}\|c(x_k)\|\leq\kappa_{\rm n}\|c(x_k)\|.}$$  
\textcolor{blue}{Then we have from \eqref{Zkbd} that}
$$\textcolor{blue}{0\leq\|Z_k^TB_kv_k\|\leq\|Z_k^T\| \|B_k\| \|v_k\|\leq \kappa_{\rm Z}\kappa_{\rm B}\kappa_{\rm n}\|c(x_k)\|.}$$ 
\textcolor{blue}{It then follows from Theorem \ref{thm3.1} that $Z_k^TB_kv_k$ tends to zero.}

Thus, by \eqref{klcg} and the definition of $g^Z_k$ in \eqref{Bzkgzkdef}, there exists an index $k_2$ such that
\begin{eqnarray}\label{lbg2}
  \|g^Z_k\|\geq\epsilon>0
\end{eqnarray}
for all $k\geq k_2$. 

From Assumptions $\mathrm{(A1)-(A3)}$, \eqref{lema2} and \eqref{kbtm} hold by Lemma \ref{lem3.2} and Lemma \ref{lem3.8}.
Then it follows from \eqref{lbg2}, \eqref{lema2}, and \eqref{kbtm} that
\begin{equation}\label{68}
 \Delta q_k\geq\kappa_{\rm btm}\kappa_{\rm qkh}\sqrt{\beta_k\|g_k^Z\|^3}\geq\kappa_{\rm btm}\kappa_{\rm qkh}\epsilon^{\frac{3}{2}}\sqrt{\beta_k}
\end{equation}
for all $k\geq k_3:={\max}\{k_1,k_2\}$.\\
Let
\begin{eqnarray}\label{sigma00}
 \bar{\beta}^0:=(\frac{\kappa_{\rm btm}\kappa_{\rm qkh}\epsilon^{\frac{3}{2}}(1-\eta_2)}{\kappa_{\rm m}(1+\mu_{\max})})^2.
\end{eqnarray}
Suppose that $\beta_k\leq\bar{\beta}^0$. Then, from \eqref{fimerror}, \eqref{mumax}, \eqref{68} and \eqref{sigma00}, we get
\begin{eqnarray*}
  1-\rho_k&=& \frac{\phi(x_k+d_k,\mu_k)-q(x_k,B_k,\mu_k,d_k)}{\Delta q_k}\\
   &\leq& \frac{\kappa_{\rm m}(1+\mu_{\max})}{\kappa_{\rm btm}\kappa_{\rm qkh}\epsilon^{\frac{3}{2}}}\sqrt{\beta_k}\leq 1-\eta_2,
\end{eqnarray*}
which implies that $\rho_k\geq\eta_2$. Therefore, the iteration must be very successful whenever $k\geq k_3$ and $\beta_k\leq\bar{\beta}^0$.

We now prove that
\begin{eqnarray}\label{gamma1}
 \beta_k>\gamma_1\bar{\beta}^0
\end{eqnarray}
for all $k\geq k_3$.

Suppose that \eqref{gamma1} is false. Let $k$ be the first iteration such that $\beta_{k+1}\leq\gamma_1\bar{\beta}^0$. Since $\gamma_1\beta_k\leq\beta_{k+1}$ follows from the Algorithm \ref{Alg2}, we have $\beta_k\leq\bar{\beta}^0$. But then, as concluded in the previous paragraph, iteration $k$ must have been very successful, and  thus $\beta_k\leq\beta_{k+1}$. This contradicts the assumption that $k$ is the first iteration such that $\beta_{k+1}\leq\gamma_1\bar{\beta}^0$. Hence, the inequality \eqref{gamma1} must hold.
So we get from \eqref{68} and \eqref{gamma1} that
\begin{eqnarray*}
  \Delta q_k\geq s_3:=\kappa_{\rm btm}\kappa_{\rm qkh}\epsilon\sqrt{\gamma_1\bar{\beta}^0\epsilon}.
\end{eqnarray*}
\textcolor{blue}{For all $k\in\mathcal{S}$ and $k\geq k_3$,}
\begin{eqnarray*}
\phi(x_k,\mu_{\max})-\phi(x_{k+1},\mu_{\max})\geq\eta_1\Delta q_k\geq\eta_1s_3.
\end{eqnarray*}
Summing the above over the set of successful iterations between iteration $k_3$ and $k$, and letting $\tau_k$ be the number of such successful iterations, we deduce that
\begin{eqnarray*}
 \phi(x_k,\mu_{\max})\leq\phi(x_{k_3},\mu_{\max})-\tau_k\eta_1s_3.
\end{eqnarray*}
Since Assumption (A1) implies that $\{\phi(x_k,\mu_{\max})\}$ is bounded from below, we obtain that $\tau_k$ is bounded. But this is impossible, since this would imply that there is an index $l$ such that $x_k=x_l$ for all $k\geq l$ after which the cubic regularisation parameter $\beta_k$ must converge to zero, which is not compatible with Lemma \ref{sigmabounded}. Therefore, our initial hypothesis \eqref{klcg} does not hold and the conclusion of this theorem is true.
\end{proof}

A stronger conclusion can be obtained under the following assumption.\\
\indent $\bf(A5)$~~~$Z^T_kg_k$ is uniformly continuous on the sequence of iterates $\{x_k\}$, namely, $\|Z^T_{t_i}g_{t_i}-Z^T_{l_i}g_{l_i}\|\rightarrow0, ~{\rm whenever}~\|x_{t_i}-x_{l_i}\|\rightarrow0,~i\rightarrow\infty$, where $\{x_{t_i}\}$ and $\{x_{l_i}\}$ are subsequences of $\{x_k\}$.

\begin{theorem}\label{thm3.4}
Let Assumptions $\mathrm{(A1)-(A5)}$ hold. Then
\begin{eqnarray*}
 \lim_{k\rightarrow\infty}\|Z^T_kg_k\|=0.
\end{eqnarray*}
\end{theorem}
\begin{proof}
Suppose that there is an infinite subsequence $\{t_i\}\subset\mathcal{S}$ such that
\begin{eqnarray}\label{3epsilon}
 \|Z^T_{t_i}g_{t_i}\|\geq3\epsilon
\end{eqnarray}
for some $\epsilon>0$ and for all $i$. Theorem \ref{thm3.3} ensures that for each $t_i$, there is a first successful iteration $l_i>t_i$ such that $\|Z^T_{l_i}g_{l_i}\|<2\epsilon$. Thus $\{l_i\}\subset\mathcal{S}$, and for all $i$ and $t_i\leq k<l_i$,
\begin{eqnarray}\label{epsilon}
  \|Z^T_kg_k\|\geq2\epsilon,~{\rm and}~\|Z^T_{l_i}g_{l_i}\|<2\epsilon.
\end{eqnarray}
Similar to \eqref{lbg2}, we have that
\begin{equation}\label{gnepsilon}
\|g^Z_k\|\geq\epsilon>0.
\end{equation}
Let $\mathcal{K}:=\{k\in\mathcal{S}:t_i\leq k<l_i\}$, where $\{t_i\}$ and $\{l_i\}$ are defined above. Since $\mathcal{K}\in\mathcal{S}$, we can get from \eqref{lema2}, Assumption (A3), \eqref{kbtm} and \eqref{gnepsilon}  that
\begin{eqnarray}\label{tmd}
 \Delta q_k\geq \kappa_{\rm btm}\kappa_{\rm qkh}\epsilon\sqrt{\beta_k\epsilon},~k\in\mathcal{K}.
\end{eqnarray}
From \eqref{varphi2}, \eqref{gnepsilon} and \eqref{tmd},  we have for all $~k\in\mathcal{K}$ that
\begin{eqnarray*}
 \Phi(x_k,\mu_k)-\Phi(x_{k+1},\mu_{k+1})\geq\frac{\eta_1\kappa_{\rm btm}\kappa_{\rm qkh}\epsilon^{\frac{3}{2}}}{\mu_k}\sqrt{\beta_k}.
\end{eqnarray*}
It then follows from \eqref{mumax} that
\begin{eqnarray}\label{phixk}
 \Phi(x_k,\mu_k)-\Phi(x_{k+1},\mu_{k+1})\geq \frac{\eta_1\kappa_{\rm btm}\kappa_{\rm qkh}\epsilon^{\frac{3}{2}}}{\mu_{\max}}\sqrt{\beta_k}
\end{eqnarray}
for all sufficiently large $k\in\mathcal{K}$. From \eqref{epsilon} and \eqref{phixk}, we have that \eqref{d39} is satisfied when $\mathcal{I}=\mathcal{K}$, and thus
\begin{eqnarray*}
 \Phi(x_k,\mu_k)-\Phi(x_{k+1},\mu_{k+1})\geq \frac{\eta_1\kappa_{\rm btm}\kappa_{\rm qkh}\epsilon^{\frac{3}{2}}}{\mu_{\max}M_{\rm d}}  \|d_k\|
\end{eqnarray*}
for all $t_i\leq k<l_i$, $k\in\mathcal{S}$, $i$ sufficiently large. Summing the above between iteration $t_i$ and $l_i$, and employing the triangle inequality, we deduce that,
\begin{eqnarray}\label{xli}
 &&\frac{\mu_{\max}M_{\rm d}}{\eta_1\kappa_{\rm btm}\kappa_{\rm qkh}\epsilon^{\frac{3}{2}}} (\Phi(x_{t_i},\mu_{t_i})-\Phi(x_{l_i},\mu_{l_i}))\nonumber\\
 &\geq&\sum^{l_i-1}_{k=t_i,k\in\mathcal{S}}\|d_k\|  = \sum^{l_i-1}_{k=t_i,k\in\mathcal{S}}\|x_{k+1}-x_k\| =\|x_{l_i}-x_{t_i}\|
\end{eqnarray}
for all $i$ sufficiently large. Since $\{\Phi(x_k,\mu_k)\}$ is convergent, we have
\begin{eqnarray*}
 \lim_{i\rightarrow\infty}\|\Phi(x_{t_i},\mu_{t_i})-\Phi(x_{l_i},\mu_{l_i})\|=0,
\end{eqnarray*}
and then \eqref{xli} implies that
\begin{eqnarray*}
 \lim_{i\rightarrow\infty}\|x_{l_i}-x_{t_i}\|=0.
\end{eqnarray*}
Assumption (A5) implies that
\begin{eqnarray}\label{Nti}
 \lim_{i\rightarrow\infty}\|Z^T_{l_i}g_{l_i}-Z^T_{t_i}g_{t_i}\|=0.
\end{eqnarray}
But from \eqref{3epsilon} and \eqref{epsilon}, we have
\begin{eqnarray*}
 \|Z^T_{l_i}g_{l_i}-Z^T_{t_i}g_{t_i}\|\geq\|Z^T_{t_i}g_{t_i}\|-\|Z^T_{l_i}g_{l_i}\|\geq\epsilon,
\end{eqnarray*}
which is a contradiction to \eqref{Nti}. Hence, the theorem is proved.
\end{proof}
\section{Numerical Results}\label{sec4}
 In this section, we present the numerical results of $\mathrm{SSARC_{q}K}$ (Algorithm \ref{Alg2}) which have been performed on a desktop with Intel(R) Core(TM) i5-8265U CPU @ 1.60GHz. $\mathrm{SSARC_{q}K}$ is implemented under MATLAB(R2022a).

In the implementation of $\mathrm{SSARC_{q}K}$, $B_k$ is the Hessian of the Lagrangian function and the vector of Lagrange multipliers is computed
 by solving
 $\underset{s\in\mathbb{R}^p}{\min} ~\|g_k-J_k^Ts\|.$

Other parameters are set as follows.
$$\gamma_1=0.1,\gamma_2=5,\eta_1 = 0.01, \eta_2 = 0.75,\tau_1=2, \tau_2=1,\nu=10^{-4}, \xi=0.1, \zeta=0.01.$$
The stop tolerance for $\mathrm{SSARC_{q}K}$  is
 \begin{eqnarray*}
 \mathrm{Res}=\mathrm{max}\{\|Z^T_kg_k\|,\|c(x_k)\|\}\leq \epsilon.
\end{eqnarray*}

In this section, $N_{it}$ denotes the number of iterations, $N_f$, $N_c$, and $N_g$ are the number of computation of the objective function,  the number of computation of the constraint function, and the number of computation of the gradient of the objective function, respectively. $n$ denotes the number of variables. $p$ denotes the number of equality constraints.

We first test $\mathrm{SSARC_{q}K}$ on a set of equality constrained optimization problems from CUTEst collection \cite{Gould2015cutest}. The setting of termination criteria is $\mathrm{Res}\leq \epsilon=10^{-8}$. The numerical results are reported in Table \ref{table5.1}, where the CPU-time is counted in seconds. 

\textcolor{blue}{Although the CUTEst test problems in Table \ref{table5.1} are not so large in scale, they represent most of numerical difficulties in practice, such as badly scaled objective and constraint functions, badly scaled variables, ill-conditioned optimization problems, non-regular solutions at points where the constraint qualification is not satisfied, different local solutions, and infinitely many solutions. Therefore, we use these problems to test the robustness of the algorithm. The numerical results show that $\mathrm{SSARC_qK}$ is robust and efficient when solving these small- and medium-sized problems.
Then we compare the performance of $\mathrm{SSARC_qK}$ and algorithms in \cite{2002JOTA,Liu2011SIOPT,Chen2020COAP}, respectively.}
\subsection*{\textcolor{blue}{(I) Comparison with Algorithm A-SQP in \cite{2002JOTA}}}
\textcolor{blue}{Algorithm A-SQP in \cite{2002JOTA} is an adaptive algorithm for constrained least-squares problems. A-SQP is also compared with the method in \cite{Bergou2021SIAMJSC}, which is a Levenberg--Marquardt algorithm for nonlinear least squares subject to nonlinear equality constraints.
We compare the performance of A-SQP and $\mathrm{SSARC_qK}$ on two sets of test problems; one is a parameter estimation problem (Example 4.1 in \cite{2002JOTA}); the other is from Section 4 in \cite{2002JOTA} (Examples 4.2 in \cite{2002JOTA}). The stop tolerance $\epsilon=10^{-5}$ is the same as in \cite{2002JOTA} in the comparison.}

{\textcolor{blue}{\bf (i) Example 4.1 in \cite{2002JOTA}}}.\\
\textcolor{blue}{\indent We start the comparison from Example 4.1 in \cite{2002JOTA} as follows.}  
\begin{subequations}
\begin{align*}
\textcolor{blue}{{\underset {x\in\mathbb{R}^3}{{\rm minimize}}}}  &\ \ \  \textcolor{blue}{\frac{1}{2} \sum_{i=1}^{\hat{n}} r_{i}^{2}(x)} \ \\
\mbox{\textcolor{blue}{subject to}} &\ \ \  \textcolor{blue}{\alpha + \beta + \gamma^2 = 0,}
\end{align*}
\end{subequations}
\textcolor{blue}{where the parameter $x = (\alpha, \beta, \gamma)^T$, the residual}
$$ \textcolor{blue}{r_{i}(x) := y_{i} - \left( e^{-\alpha t_{i}} - e^{-\beta t_{i}} - \gamma e^{-t_{i}} - e^{-10 t_{i}} \right),}$$
\textcolor{blue}{and $ y_1, y_2, \ldots, y_{\hat{n}} $ are observations of some parameter estimation problem, that is,}
$$
\textcolor{blue}{y_i = e^{-\alpha t_i} - e^{-\beta t_i} - \gamma e^{-t_i} - e^{-10 t_i} + \epsilon_i,}
$$
\textcolor{blue}{where $ t_i = i / {\hat{n}} $ and $ \{\epsilon_i\}_{i=1}^{\hat{n}} $ are independent Gaussian random variables with mean $ 0 $ and variance $ \sigma^2 $.} 

\textcolor{blue}{In the comparison, the experimental setup is the same as that in \cite{2002JOTA}, that is, 
the simulation data $ y_i $ is generated using the true parameter values:}
$$\textcolor{blue}{\alpha = 1.0, \quad \beta = -2.0, \quad \gamma = 1.0.}$$
\textcolor{blue}{Three cases for the noise standard deviation $ \sigma $ are tested:}
$$\textcolor{blue}{\sigma = 1.0, \quad \sigma = 0.1, \quad \sigma = 0.01.}$$

\textcolor{blue}{For all cases, the initial point is set as: $x_0 = [0.8, -2.3, 0.9]^T.$}

\textcolor{blue}{The number of observations $ \hat{n} $ is varied as: 
$\hat{n} = 10, \quad 10^2, \quad 10^3, \quad 10^4.$}

\textcolor{blue}{Since $ \{\epsilon_i\}_{i=1}^n $ are independent Gaussian random variables, we run $\mathrm{SSARC_{q}K}$ 50 times and select the best results to compare with A-SQP. The numerical comparison results are reported in Table \ref{2002JOTAcomparison} in terms of the number of function evaluations $ N_f $ and gradient evaluations $ N_g $.}

\textcolor{blue}{From the numerical experiment results, we can observe that A-SQP outperforms $\mathrm{SSARC_qK}$. This is because A-SQP exploits and utilizes the  structure properties of least-squares problems, whereas $\mathrm{SSARC_qK}$ does not use these properties.}

{\textcolor{blue}{{\bf (ii) Examples 4.2 in \cite{2002JOTA}}}.\\
\textcolor{blue}{\indent We compare $\mathrm{SSARC_{q}K}$ with A-SQP on some standard test examples. A detailed description of these examples can be found in \cite{Gould2015cutest} and references therein. All the problems are numbered as in these references. For those problems having inequality constraints, only the constraints active at the solution are included. Linearly constrained problems have been excluded. The experimental setup is also the same as that in \cite{2002JOTA}. We include the corresponding numerical results of A-SQP. The comparison numerical results are reported in Table \ref{comparison2002ex42}.}

\textcolor{blue}{To display the performance based on the numerical results in Table \ref{comparison2002ex42} visually, we use the logarithmic performance profiles \cite{Dolan} (see Figure \ref{ngnf2002jota}). From Table \ref{comparison2002ex42} and Figure \ref{ngnf2002jota}, it can be seen that $\mathrm{SSARC_{q}K}$ performs better than Algorithm A-SQP in \cite{2002JOTA} for the given problems.}
\begin{figure}[htbp]
	\centering
	\includegraphics[width=5.55cm,height=4.8cm]{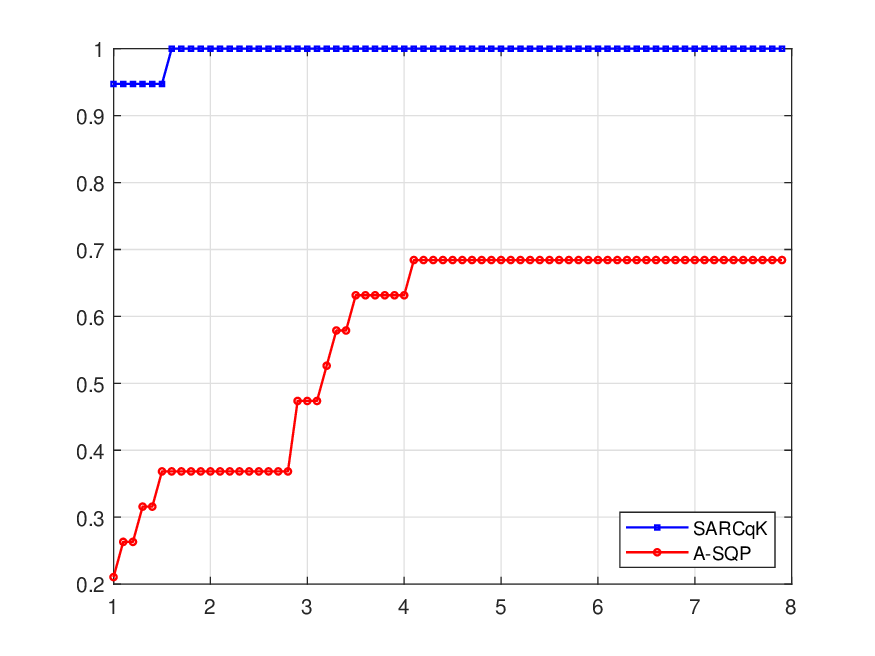}
	\quad
	\includegraphics[width=5.55cm,height=4.8cm]{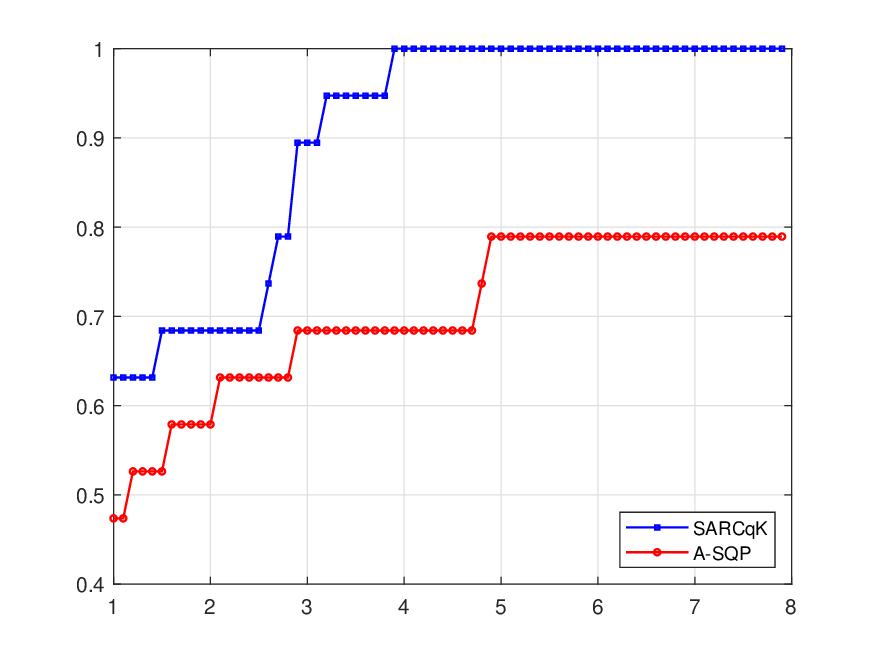} \\[0pt]
	\caption{\textcolor{blue}{Performance profiles based on numbers of objective function evaluation $N_f$ (left) and numbers of objective function's gradient   
 evaluation $N_g$ (right).}}\label{ngnf2002jota}
\end{figure}
\subsection*{\textcolor{blue}{(II) Comparison with Algorithm 2.2 in \cite{Liu2011SIOPT}}}
\textcolor{blue}{Algorithm 2.2 in \cite{Liu2011SIOPT} is a sequential quadratic programming method without using a penalty function or
a filter for solving nonlinear equality constrained optimization. It is robust and efficient when solving some small- and medium-sized problems.
The comparison results are listed in Table \ref{Liu2011comparison}, where $N_j$ denotes the number of Jacobian matrix $J(x)$ evaluations. 
Figure \ref{nfngliu} displays the performance based on the numerical results in Table \ref{Liu2011comparison}. The stop tolerance $\epsilon=10^{-5}$ is the same as in \cite{Liu2011SIOPT} in the comparison.}
\begin{figure}[htbp]
	\centering
	\includegraphics[width=5.55cm,height=4.8cm]{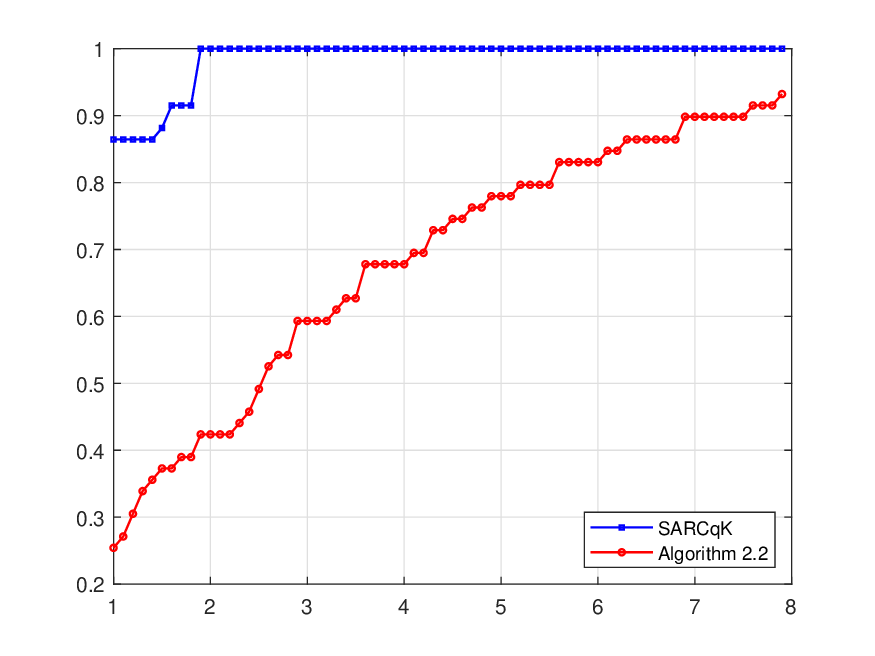}
	\quad
	\includegraphics[width=5.55cm,height=4.8cm]{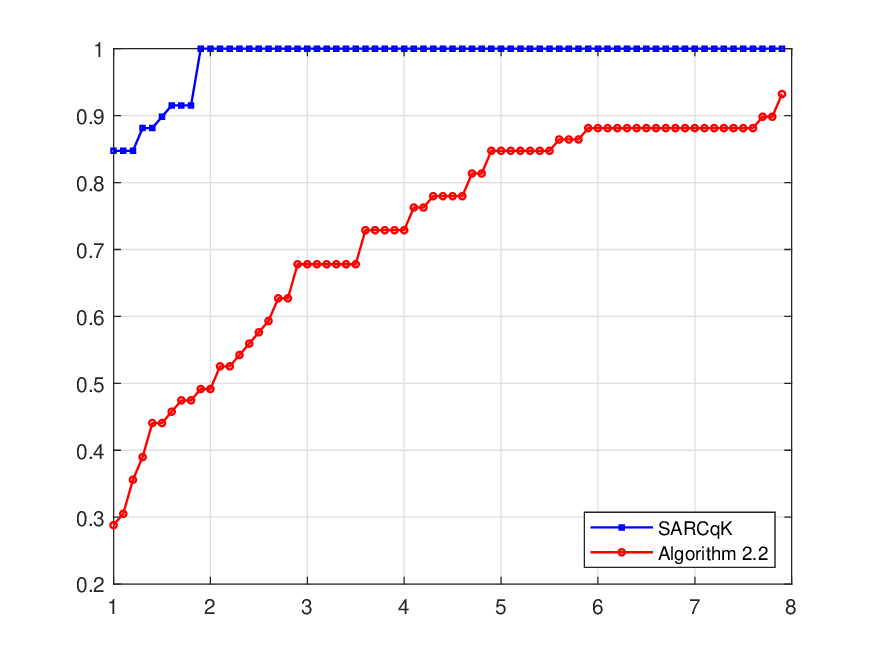} \\[0pt]
	\caption{\textcolor{blue}{Performance profiles on $N_f=N_c$ (left) and $N_g=N_j$ (right).}}\label{nfngliu}
\end{figure}

\textcolor{blue}{From these comparison results, it can be seen that $\mathrm{SSARC_qK}$ is competitive with Algorithm 2.2 \cite{Liu2011SIOPT} for the given problems.}

\subsection*{\textcolor{blue}{(III) Comparison with Algorithm 2.1 in \cite{Chen2020COAP}}}
\textcolor{blue}{We first test $\mathrm{SSARC_qK}$ for the following two-dimensional problem given by Powell \cite{PowellSIAMReview1986}.}
\begin{subequations}\label{maratosexm}
\begin{align}
\textcolor{blue}{{\underset {x\in\mathbb{R}^2}{{\rm minimize}}}}   &\ \ \  \textcolor{blue}{f(x)=-x_1 +\rho (x_1^2+x_2^2-1)}\ \\
\mbox{\textcolor{blue}{subject to}} &\ \ \  \textcolor{blue}{c(x)=x_1^2+x_2^2-1= 0.}
\end{align}
\end{subequations}
\textcolor{blue}{This problem is used to illustrate the Maratos effect in \cite{Chen2020COAP} with $\rho=2$. The Maratos effect is very serious for $\rho\geq 10$. We test $\mathrm{SSARC_qK}$ for $\rho=2, 10, 100, 1000$, while we compare $\mathrm{SSARC_qK}$ with Algorithm 2.1 in \cite{Chen2020COAP} and its variant which uses objective function in the acceptance criterion in \cite{Chen2020COAP}. The initial points and the setting of termination criteria (${\rm Res}\leq\epsilon=10^{-10}$) are the same as those in \cite{Chen2020COAP}. The comparison results are reported in Table \ref{tabchenobj} and Table \ref{tabchenlag}.}

\textcolor{blue}{In Table \ref{tabchenobj}, F denotes fail when $N_i>1000$. $\mathrm{SSARC_qK}$ outperforms Algorithm 2.1 in \cite{Chen2020COAP} for $\rho=2$. For $\rho=10$, the two algorithms perform basically the same. Algorithm 2.1 in \cite{Chen2020COAP} is significantly superior to $\mathrm{SSARC_qK}$ for $\rho=100$. Both $\mathrm{SSARC_qK}$ and Algorithm 2.1 in \cite{Chen2020COAP} fail for $\rho=1000$ when using the objective function in the acceptance criterion (thus not listed in the table). This may be because both algorithms suffer from the severe Maratos effect.}

\textcolor{blue}{Inspired by Algorithm 2.1 in \cite{Chen2020COAP}, we tried to replace the original objective function with a Lagrangian function in $\mathrm{SSARC_qK}$. The numerical results are significantly improved and almost the same as the performance of Algorithm 2.1 (see Table \ref{tabchenlag}). This indicates that replacing the original objective function with a Lagrangian function may alleviate the Marotos effect of the algorithm, but this requires theoretical analysis and proof. Such analysis and proof are not necessarily straightforward to obtain, which will also be a focus of our future research.} 

\textcolor{blue}{Then we include and compare the corresponding numerical results of Algorithm 2.1 in \cite{Chen2020COAP} for a set of standard test problems from \cite{Gould2015cutest}. The comparison numerical results are reported in Table \ref{tab2020chen}, where $I$ means that Algorithm 2.1 \cite{Chen2020COAP} gets an infeasible stationary point in the case $n=m$.}
\textcolor{blue}{For comparison, the dimensions of the test problems are set to be the same as those in \cite{Chen2020COAP}.}

\textcolor{blue}{Furthermore, to show the comparison of the performance based on the numerical results in Table \ref{tab2020chen} visually, we also use the logarithmic performance profiles \cite{Dolan} (see Figure \ref{nfnc2020chen}). From Table \ref{tab2020chen} and Figure \ref{nfnc2020chen}, it can be seen that $\mathrm{SSARC_qK}$ performs better than Algorithm 2.1 in \cite{Chen2020COAP} for these standard test problems.}
\begin{figure}[htbp]
	\centering
	\includegraphics[width=5.55cm,height=4.8cm]{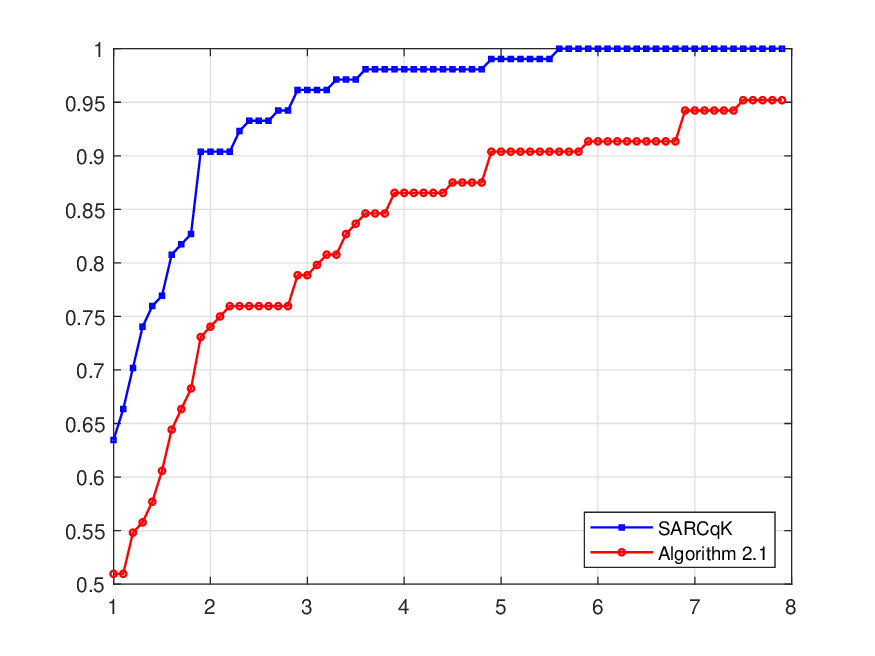}
	\quad
	\includegraphics[width=5.55cm,height=4.8cm]{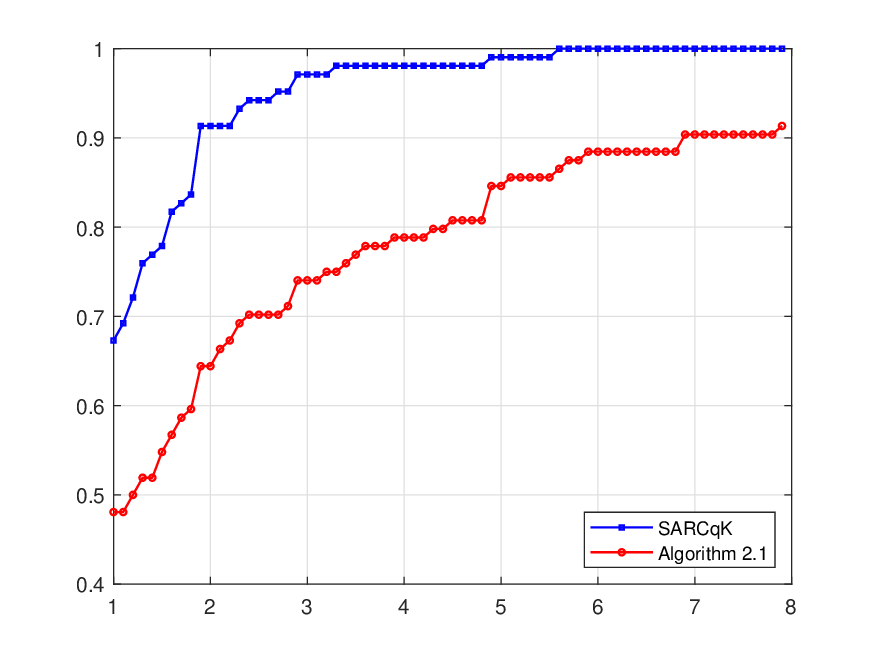} \\[0pt]
	\caption{\textcolor{blue}{Performance profiles on Performance profiles on $N_f$ (left) and $N_c$ (right).}}\label{nfnc2020chen}
\end{figure}

\section{Conclusions}
\label{sec6}
 We have introduced a scalable sequential adaptive cubic regularization algorithm ($\mathrm{SSARC_qK}$) for nonlinear equality constrained optimization. \textcolor{blue}{Composite-step methods are employed to decompose the trial step into the sum of the vertical step and the horizontal step.} The horizontal step is obtained by solving an \textcolor{blue}{ARC subproblem's necessary and sufficient optimality conditions system inexactly. Specifically,} we employ CG-Lanczos procedure with shifts to this end. \textcolor{blue}{This procedure is analogous to Algorithm 2 in \cite{Dussault19}.} Then the exact penalty function is used as the merit function to decide whether the trial step is accepted. Under some suitable assumptions, we prove the global \textcolor{blue}{convergence to first-order critical points.} \textcolor{blue}{Numerical tests and some comparison results are given to demonstrate the performance of $\mathrm{SSARC_qK}$ for small- and medium-sized problems.  
 There are two key distinctions between $\mathrm{SSARC_qK}$ and the two algorithms in \cite{Pei32,Pei33}. First, the approaches to solving ARC subproblems differ: the two previous algorithms solve ARC subproblems in line with optimality conditions \eqref{TKM}, whereas $\mathrm{SSARC_qK}$ adopts an inexact optimality conditions \eqref{TKm} for this purpose. Second, the convergence analysis varies: the two previous algorithms rely on optimality conditions to establish convergence, while $\mathrm{SSARC_qK}$ can only leverage inexact optimality conditions for its convergence proof.}
 
 \textcolor{blue}{In future research, we shall implement and test $\mathrm{SSARC_qK}$ carefully in the large-scale case. For instance, we can consider matrix-free methods similar to those proposed in \cite{Byrd2010MP,Heinkenschloss2014SIOPT,Bergou2021SIAMJSC}, which exhibit high computational efficiency when solving large-scale problems.     
 In addition, $\mathrm{SSARC_qK}$ may suffer from the Maratos effect as discussed in the comparison with Algorithm 2.1 in \cite{Chen2020COAP}. We need to study the corresponding techniques to overcome the Maratos effect. And the worst-case function- and derivative-evaluation complexity of the algorithm and the local convergence are also the issues that we plan to address in our future research.}

\begin{acknowledgement}
The authors are very grateful to the anonymous referees and editors for their helpful comments on this paper.
The authors gratefully acknowledge the supports of the National Natural Science Foundation of China (12071133), Natural Science Foundation of Henan Province (252300421993) and Key Scientific Research Project for Colleges and Universities in Henan Province (25B110005). 
\end{acknowledgement}

\begin{dataa}
 There is no data generated or analysed in this paper.
\end{dataa}

\section*{Declarations}
\small  The authors declare that they have no conflict of interest.

\section*{Appendix: Tables in Numerical Results}

\begin{center}
{
\begin{longtable}{cccccccc}
\caption{Numerical results of $\mathrm{SSARC_{q}K}$ for CUTEst problems}\label{table5.1}
  \endfirsthead   
  \multicolumn{8}{c}{Table \ref{table5.1} continued}\\
  \hline
\multirow{2}{*}{Problem}  & \multicolumn{2}{c}{Dimension} &\multirow{2}{*}{$N_{it}$}  &\multirow{2}{*}{CPU-time} & \multirow{2}{*}{Res}  &\multirow{2}{*}{$N_f$}  &\multirow{2}{*}{$N_g$ }\\
                          \cmidrule{2-3}

                          &$n$&$p$                                                         \\
  \hline
  \endhead
  \hline
\multirow{2}{*}{Problem}  & \multicolumn{2}{c}{Dimension} &\multirow{2}{*}{$N_{it}$}  &\multirow{2}{*}{CPU-time} & \multirow{2}{*}{Res}  &\multirow{2}{*}{$N_f$}  &\multirow{2}{*}{$N_g$ }\\
                          \cmidrule{2-3}

                          &$n$&$p$                       \\
\hline
AIRCRFTA    &8    &5        &3  &0.002129       &2.2215e-16   &4    &3       \\
 ARGTRIG    &200  &200      &3  &0.096873       &8.1517e-10   &4    &3        \\
 ARTIF      &5002 &5000     &6  &70.187258      &1.4900e-09   &7    &6        \\
 AUG2D      &703  &300      &5  &0.344238       &2.4374e-11   &6    &5       \\
 AUG2DC     &703  &300      &11 &0.919471       &1.2970e-09   &12   &11       \\
 BDVALUE    &5002 &5000     &2  &15.120100      &7.3096e-09   &3    &2        \\
 BDVALUES   &10002&10000    &12 &95.128044      &2.2229e-10   &13   &12       \\
 BOOTH      &2    &2        &2  &0.002069       &4.4409e-16   &3    &2        \\
 BRATU2D    &5184 &4900     &2  &41.399009      &7.4054e-09   &3    &2        \\
 BRATU2DT   &5184 &4900     &3  &49.120426      &3.6233e-14   &4    &3        \\
 BRATU3D    &4913 &3375     &3  &30.817134      &1.6127e-13   &4    &3        \\
 BROWNALE   &200  &200      &8  &0.171668       &2.0178e-10   &9    &8       \\
 BROYDN3D   &5000 &5000     &4  &50.936617      &1.0741e-09   &5    &4        \\
  BT1       &2    &1        &5  &0.002108       &7.1663e-09   &6    &5        \\
  BT2       &3    &1        &11 &0.009406       &9.0428e-14   &12   &11       \\
  BT3       &5    &3        &4  &0.004625       &8.9598e-10   &5    &4        \\
  BT4       &3    &2        &6  &0.004607       &7.2720e-14   &7    &6        \\
  BT5       &3    &2        &4  &0.002099       &8.9746e-09   &5    &4        \\
  BT6       &5    &2        &10 &0.005238       &5.4852e-09   &11   &10       \\
  BT7       &5    &3        &12 &0.005609       &3.0271e-14   &13   &12       \\
  BT8       &5    &2        &14 &0.004503       &5.2684e-09   &15   &14       \\
  BT9       &4    &2        &8  &0.002767       &5.2087e-13   &9    &8       \\
  BT10      &2    &2        &6  &0.003610       &4.4039e-09   &7    &6        \\
  BT11      &5    &3        &7  &0.003306       &1.4334e-11   &8    &7        \\
  BT12      &5    &3        &4  &0.002481       &2.3846e-13   &5    &4        \\
  BYRDSPHR  &3    &2        &8  &0.002721       &1.6387e-13   &9    &8       \\
  CBRATU2D  &3200 &2888     &3  &11.201362      &1.5569e-15   &4    &3        \\
  CBRATU3D  &3456 &2000     &3  &8.366288       &3.8684e-13   &4    &3        \\
  CHANDHEU  &500  &500      &15 &1.384879       &9.5858e-09   &16   &15       \\
  CLUSTER   &2    &2        &9  &0.007128       &9.7423e-12   &10   &9        \\
  CUBENE    &2    &2        &5  &0.004295       &0.0000e+0    &6    &5        \\
  DECONVNE  &63   &40       &2  &0.009853       &1.2385e-15   &3    &2        \\
  DTOC1L    &5998 &3996     &9  &353.362691     &1.6970e-09   &10   &9       \\
  DTOC2     &5998 &3996     &7  &230.113825     &1.5442e-10   &8    &7        \\
  DTOC3     &4499 &2998     &5  &63.250006      &2.8069e-10   &6    &5       \\
  DTOC4     &4499 &2998     &2  &25.196262      &4.9860e-10   &3    &2        \\
  DTOC5     &9999 &4999     &2  &122.165166     &3.9108e-09   &3    &2        \\
  EIGENA2   &2550 &1275     &4  &12.710827      &1.1494e-10   &5    &4       \\
  EIGENAU   &2550 &2550     &6  &9.275353       &1.4071e-12   &7    &6        \\
  EIGENB2   &6    &3        &12 &0.004896       &1.6610e-10   &15   &14       \\
  EIGENC2   &2652 &1326     &10 &101.078448     &6.0021e-09   &11   &10      \\
  FLOSP2TL  &2883 &2763     &5  &12.585847      &5.3968e-11   &6    &5       \\
  GENHS28   &10   &8        &3  &0.001739       &6.4625e-11   &4    &3        \\
  GOTTFR    &2    &2        &7  &0.005328       &1.1392e-09   &8    &7        \\
  GRIDNETB  &7564 &3844     &8  &267.990458     &1.0438e-09   &9    &8        \\
  GRIDNETE  &7564 &3844     &8  &284.988885     &5.8997e-10   &9    &8        \\
  HAGER1    &10001 &5000    &2  &120.068329     &9.2810e-09   &3    &2        \\
  HAGER2    &10001 &5000    &3  &175.863774     &1.5798e-09   &4    &3        \\
  HAGER3    &10001 &5000    &3  &177.415412     &1.4325e-09   &4    &3        \\
  HATFLDF   &3    &3        &6  &0.004371       &2.8311e-12   &7    &6        \\
  HATFLDG   &25   &25       &8  &0.007073       &4.7181e-15   &9    &8        \\
  HEART6    &6    &6        &28 &0.012608       &6.2690e-09   &29   &28       \\
  HEART8    &8    &8        &11 &0.007463       &2.1299e-14   &12   &11       \\
  HIMMELBA  &2    &2        &1  &0.000692       &0.0000e+0    &2    &1        \\
  HIMMELBC  &2    &2        &5  &0.002564       &0.0000e+0    &6    &5        \\
  HIMMELBE  &3    &3        &2  &0.001286       &0.0000e+0    &3    &2        \\
  HS6       &2    &1        &7  &0.002632       &5.9686e-13   &8    &7        \\
  HS7       &2    &1        &7  &0.003295       &1.5913e-10   &8    &7        \\
  HS8       &2    &2        &4  &0.002254       &2.2061e-12   &5    &4        \\
  HS9       &2    &1        &6  &0.002361       &1.7298e-09   &7    &6        \\
  HS26      &3    &1        &16 &0.006608       &6.9605e-09   &17   &16      \\
  HS27      &3    &1        &10 &0.003562       &7.4560e-09   &11   &10      \\
  HS28      &3    &1        &3  &0.001554       &7.0407e-12   &4    &3       \\
  HS39      &4    &2        &8  &0.003158       &5.6164e-11   &9    &8      \\
  HS40      &4    &3        &3  &0.001572       &3.4209e-10   &4    &3       \\
  HS42      &4    &2        &3  &0.001530       &9.3173e-09   &4    &3       \\
  HS46      &5    &2        &22 &0.011721       &7.6893e-09   &23   &22      \\
  HS47      &5    &3        &19 &0.012813       &5.4288e-09   &20   &19      \\
  HS48      &5    &2        &3  &0.001755       &8.6179e-13   &4    &3       \\
  HS49      &5    &2        &21 &0.009783       &8.3110e-09   &22   &21      \\
  HS50      &5    &3        &9  &0.005791       &3.5567e-13   &10   &9       \\
  HS51      &5    &3        &2  &0.001446       &4.3252e-09   &3    &2        \\
  HS52      &5    &3        &2  &0.001424       &3.4710e-09   &3    &2        \\
  HS56      &7    &4        &12 &0.006713       &3.5278e-09   &13   &12       \\
  HS61      &3    &2        &5  &0.002398       &1.3609e-11   &7    &6        \\
  \textcolor{blue}{HS65}      &3    &1        &12 &0.004097       &3.1640e-12   &14   &13      \\
  HS77      &5    &2        &9  &0.004594       &8.9764e-10   &11   &10       \\
  HS78      &5    &3        &10 &0.004134       &2.8086e-09   &11   &10       \\
  HS79      &5    &3        &5  &0.002961       &1.1567e-11   &6    &5        \\
  HS100LNP  &7    &2        &6  &0.003542       &2.4747e-12   &9    &8        \\
  HS111LNP  &10   &3        &12 &0.007043       &1.3675e-10   &13   &12       \\
  HYDCAR20  &99   &99       &9  &0.109790       &1.3560e-11   &10   &9        \\
  HYDCAR6   &29   &29       &5  &0.008142       &5.9618e-14   &6    &5          \\
  HYPCIR    &2    &2        &4  &0.002165       &1.9989e-10   &5    &4          \\
  INTEGREQ  &502  &500      &3  &0.482701       &3.9141e-14   &4    &3          \\
  JUNKTURN  &1010 &700      &10 &11.612983      &1.3013e-09   &11   &10           \\
  MARATOS   &2    &1        &3  &0.001913       &1.7235e-09   &4    &3           \\
  METHANB8  &31   &31       &3  &0.003778       &1.9401e-13   &4    &3             \\
  METHANL8  &31   &31       &4  &0.004632       &5.2325e-11   &5    &4           \\
  MSQRTA    &1024 &1024     &6  &1.199536       &2.5112e-14   &7    &6           \\
  MSQRTB    &1024 &1024     &6  &1.255476       &2.5602e-14   &7    &6           \\
  MWRIGHT   &5    &3        &5  &0.003397       &3.6934e-11   &6    &5            \\
  OPTCTRL3  &302  &200      &27 &3.236480       &6.4662e-09   &28   &27            \\
  OPTCTRL6  &302  &200      &22 &2.872068       &3.1631e-09   &23   &22            \\
  ORTHREGA  &517  &256      &29 &20.97357      &6.1036e-09   &30   &29            \\
  ORTHREGB  &27   &6        &3  &0.002302       &9.3454e-11   &4    &3             \\
  ORTHREGC  &1005 &500      &5  &24.67617      &2.6533e-09   &6    &5             \\
  POWELLBS  &2    &2        &12 &0.007841       &7.9716e-09   &13   &12            \\
  POWELLSQ  &2    &2        &27 &0.019013       &3.3565e-12   &28   &27            \\
  RECIPE    &3    &3        &2  &0.001219       &8.9015e-16   &3    &2             \\
  RSNBRNE   &2    &2        &12 &0.007700       &0.0000e+0    &13   &12            \\
  \textcolor{blue}{S216}      &2    &1        &7  &0.003152       &3.0016e-12   &9    &8             \\
  \textcolor{blue}{S219}      &4    &2        &15 &0.004683       &7.9166e-14   &17   &16            \\
  \textcolor{blue}{S235}      &3    &1        &12 &0.002370       &4.8635e-13   &15   &13            \\
  \textcolor{blue}{S252}      &3    &1        &12 &0.002779       &1.4067e-13   &15   &13            \\
  \textcolor{blue}{S254}      &3    &2        &5  &0.000898       &6.0863e-09   &6    &6            \\
  \textcolor{blue}{S269}      &5    &3        &5  &0.001050       &6.6138e-14   &7    &6             \\
  S316-322  &2    &1        &5  &0.003500       &5.0306e-12   &6    &5             \\
  \textcolor{blue}{S327}      &2    &1        &9  &0.001335       &7.6719e-09   &11   &10            \\
  \textcolor{blue}{S335}      &3    &2        &17 &0.001788       &2.3825e-11   &22   &18            \\
  \textcolor{blue}{S336}      &3    &2        &8  &0.001016       &5.8287e-16   &10   &9             \\ 
  \textcolor{blue}{S338}      &3    &2        &7  &0.000981       &3.4717e-13   &8    &8             \\ 
  \textcolor{blue}{S344}      &3    &1        &22 &0.002467       &4.3874e-09   &24   &23            \\ 
  \textcolor{blue}{S345}      &3    &1        &11 &0.001545       &7.4406e-12   &13   &12            \\ 
  \textcolor{blue}{S373}      &9    &6        &9  &0.001929       &4.8822e-11   &11   &10            \\ 
  \textcolor{blue}{S378}      &10   &3        &13 &0.001913       &9.8054e-10   &14   &14            \\ 
  \textcolor{blue}{S394}      &20   &1        &13 &0.003302       &5.3389e-09   &15   &14            \\ 
  \textcolor{blue}{S395}     &50   &1        &12 &0.007564       &3.2195e-09   &13   &13            \\ 
  SINVALNE  &2    &2        &2  &0.000657       &5.4133e-15   &3    &2              \\
  TRIGGER   &7    &6        &7  &0.003478       &6.2235e-10   &8    &7              \\
  \textcolor{blue}{YATP1SQ}   &2600 &2600     &4  &11.198148      &2.6160e-12   &5    &5              \\
  ZANGWIL3  &3    &3        &1  &0.000795       &6.1535e-15   &2    &1              \\
  \hline
\end{longtable}
}
\end{center}

\begin{center}
\begin{longtable}{
  >{\centering\arraybackslash}m{1.2cm}  
  >{\centering\arraybackslash}m{1.0cm}  
  >{\centering\arraybackslash}m{0.8cm}  
  >{\centering\arraybackslash}m{0.8cm}  
  >{\centering\arraybackslash}m{0.8cm}  
  >{\centering\arraybackslash}m{0.8cm}  
  >{\centering\arraybackslash}m{0.8cm}  
  >{\centering\arraybackslash}m{0.8cm}  
}
\caption{\textcolor{blue}{Comparison results with A-SQP in \cite{2002JOTA} for Example 4.1.}}\label{2002JOTAcomparison}\\
\hline
\multirow{2}{*}{Algorithm} & \multirow{2}{*}{$\hat{n}$} & \multicolumn{2}{c}{$\sigma = 1.0$} & \multicolumn{2}{c}{$\sigma = 0.1$} & \multicolumn{2}{c}{$\sigma = 0.01$} \\
\cmidrule{3-4} \cmidrule{5-6} \cmidrule{7-8}
& & $N_f$ & $N_g$ & $N_f$ & $N_g$ & $N_f$ & $N_g$ \\
\endfirsthead
\multicolumn{8}{l}{Table \ref{2002JOTAcomparison} continued}\\
\hline
\multirow{2}{*}{Method} & \multirow{2}{*}{$\hat{n}$} & \multicolumn{2}{c}{$\sigma = 1.0$} & \multicolumn{2}{c}{$\sigma = 0.1$} & \multicolumn{2}{c}{$\sigma = 0.01$} \\
\cmidrule{3-4} \cmidrule{5-6} \cmidrule{7-8}
& & $N_f$ & $N_g$ & $N_f$ & $N_g$ & $N_f$ & $N_g$ \\
\hline
\endhead
\hline
A-SQP& 10     & 26 & 10 & 7  & 3 & 7  & 3 \\
& $10^2$ & 9  & 4  & 7  & 3 & 7  & 3 \\
& $10^3$ & 7  & 3  & 7  & 3 & 7  & 3 \\
& $10^4$ & 7  & 3  & 7  & 3 & 7  & 3 \\
   & Sum & 49 & 20 & 28 & 12& 28 & 12 \\
\cmidrule{2-8}
$\mathrm{SSARC_{q}K}$& 10     & 13  & 13  & 15  & 14 & 9   & 9 \\
& $10^2$ & 11  & 10  & 10  & 10 & 15  & 14 \\
& $10^3$ & 16  & 13  & 14  & 11 & 14  & 12 \\
& $10^4$ & 15  & 13  & 19  & 15 & 18  & 14 \\
   & Sum & 55  & 49  & 58  & 50 & 56  & 49 \\
\hline
\end{longtable}
\end{center}

\begin{center}
{
\begin{longtable}{lccccccc}  %
\caption{\textcolor{blue}{Comparison results with A-SQP in \cite{2002JOTA} for Examples 4.2.}}\label{comparison2002ex42}\\
\hline
\multirow{2}{*}{Problem} & \multicolumn{3}{c}{$\mathrm{SSARC_{q}K}$} && \multicolumn{3}{c}{A-SQP} \\
\cmidrule{2-4} \cmidrule{6-8}
& $N_f$ & $N_g$ & Res && $N_f$ & $N_g$ & Res \\
\hline
\endfirsthead
\multicolumn{8}{l}{Table \ref{comparison2002ex42} continued}\\
\hline
\multirow{2}{*}{Problem} & \multicolumn{3}{c}{$\mathrm{SSARC_{q}K}$} && \multicolumn{3}{c}{A-SQP} \\
\cmidrule{2-4} \cmidrule{6-8}
& $N_f$ & $N_g$ & Res && $N_f$ & $N_g$ & Res \\
\hline
\endhead
\hline
\multirow{2}{*}{Problem} & \multicolumn{3}{c}{$\mathrm{SSARC_{q}K}$} && \multicolumn{3}{c}{A-SQP} \\
\cmidrule{2-4} \cmidrule{6-8}
& $n_f$ & $n_g$ & $\|\nabla V\|$ && $n_f$ & $n_g$ & $\|\nabla V\|$ \\
\hline
\endfoot
\hline
\endlastfoot
HS6   &  7  &  7  & 2.1e-07  && 16  & 11  & 1.1e-30 \\
HS26  &  9  &  9  & 8.1e-06  && 19  & 10  & 4.2e-06 \\
HS42  &  3  &  3  & 9.3e-09  && 105 & 18  & 8.4e-03 \\
HS47  & 16  & 15  & 6.1e-06  && 12  & 8   & 4.8e-11 \\
HS65  & 13  & 12  & 1.5e-06  && 26  & 12  & 4.2e-06 \\
HS77  & 10  & 10  & 9.0e-06  && 12  & 5   & 8.6e-06 \\
HS79  & 5   & 5   & 3.3e-07  && 10  & 5   & 8.5e-07 \\
S216  & 10  & 9   & 1.0e-11  && 26  & 18  & 2.0e-08 \\
S235  & 15  & 13  & 5.0e-08  && 19  & 10  & 5.5e-06 \\
\textcolor{blue}{S249}  & 10  & 10  & 3.8e-06  && 10  & 4   & 1.5e-06 \\
S252  & 15  & 13  & 5.0e-08  && 33  & 17  & 9.1e-06 \\
S316  &  5  & 5   & 7.6e-11  && 68  & 30  & 7.9e-06 \\
S317  & 7   & 6   & 1.1e-06  && 78  & 46  & 9.9e-06 \\
S318  &  6  & 6   & 1.0e-06  && 84  & 49  & 8.7e-06 \\
S327  & 11  & 10  & 7.6e-09  && 12  & 5   & 8.7e-07 \\
\textcolor{blue}{S337}  &  5  & 5   & 5.3e-07  && 48  & 15  & 9.8e-06 \\
S344  & 15  & 15  & 5.1e-06  && 15  & 7   & 4.5e-06 \\
S345  & 27  & 22  & 5.3e-06  && 27  & 12  & 7.7e-06 \\
S373  & 12  & 11  & 8.4e-07  && 62  & 32  & 1.1e-09 \\
\hline
Sum  & 201  & 186 &          && 682 & 314 &
\end{longtable}
}
\end{center}

\begin{center}
	{
		\begin{longtable}{lcccccccc}
			\caption{\textcolor{blue}{Comparison results with Algorithm 2.2 in \cite{Liu2011SIOPT}.}}\label{Liu2011comparison}\\
			\endfirsthead
			\multicolumn{9}{l}{Table \ref{Liu2011comparison} continued}\\
			\hline
			\multirow{2}{*}{Problem} & \multicolumn{2}{c}{Dimension}  && \multicolumn{2}{c}{Algorithm 2.2 \cite{Liu2011SIOPT}} &\multirow{2}{*}{ } & \multicolumn{2}{c}{$\mathrm{SSARC_qK}$}\\
			\cmidrule{2-3}
			\cmidrule{5-6}
			\cmidrule{8-9}
			&$n$&$p$  & &$N_f$,$N_c$&$N_g$,$N_j$ & &$N_f$,$N_c$&$N_g$,$N_j$ \\
			\hline
			\endhead
			\hline
			\multirow{2}{*}{Problem} & \multicolumn{2}{c}{Dimension}  && \multicolumn{2}{c}{Algorithm 2.2 \cite{Liu2011SIOPT}} && \multicolumn{2}{c}{$\mathrm{SSARC_qK}$}\\
			\cmidrule{2-3}
			\cmidrule{5-6}
			\cmidrule{8-9}
			&$n$&$p$  & &$N_f$,$N_c$&$N_g$,$N_j$ & &$N_f$,$N_c$&$N_g$,$N_j$ \\
			\hline
			AIRCRFTA  &8 &5  &&3  &3  &&2  &2\\
			BDVALUE  &502 &500  &&3  &2  &&2  &2\\
			BDVALUES  &12 &10  &&10  &10  &&13  &13\\
			BOOTH  &2 &2  &&2  &2  &&2  &2\\
			BRATU2D  &484 &400  &&2  &2  &&3  &3\\
			BRATU2DT  &484 &400  &&2  &2  &&3  &3\\
			BRATU3D  &1000 &512  &&2  &2  &&3  &3\\
			BT1  &2 &1  &&11  &8  &&6  &6\\
			BT3  &5 &3  &&8  &8  &&3  &3\\
			BT4  &3 &2  &&14  &14  &&6  &6\\
			BT5  &3 &2  &&9  &9  &&5  &5\\
			BT6  &5 &2  &&30  &29  &&10  &10\\
			BT8  &5 &2  &&11  &11  &&10  &10\\
			BT9  &4 &2  &&57  &41  &&8  &8\\
			BT10  &2 &3  &&8  &8  &&7  &7\\
			BT11  &5 &3  &&13  &13  &&7  &7\\
			BT12  &5 &4  &&9  &8  &&4  &4\\
			CBRATU2D  &512 &392  &&2  &2  &&3  &3\\
			CBRATU3D  &686 &250  &&3  &3  &&4  &4\\
			CLUSTER  &2 &2  &&8  &8  &&8  &8\\
			CUBENE  &2 &2  &&12  &5  &&6  &5\\
			DECONVNE  &61 &40  &&4  &4  &&2  &2\\
			GENHS28  &10 &8  &&9  &9  &&2  &2\\
			GOTTFR  &2 &2  &&9  &6  &&7  &7\\
			HATFLDG  &25 &25  &&25  &7  &&7  &7\\
			HIMMELBA  &2 &2  &&2  &2  &&2  &2\\
			HIMMELBC  &2 &2  &&7  &6  &&5  &5\\
			HIMMELBE  &3 &3  &&3  &3  &&3  &3\\
			HS6  &2 &1  &&14  &11  &&7  &7\\
			HS7  &2 &1  &&12  &12  &&7  &7\\
			HS8  &2 &2  &&6  &5  &&5  &4\\
			HS9  &2 &1  &&7  &7  &&7  &6\\
			HS26  &3 &1  &&36  &26  &&9  &9\\
			HS28  &3 &1  &&10  &9  &&3  &3\\
			HS39  &4 &2  &&57  &41  &&8  &8\\
			HS40  &4 &3  &&7  &7  &&3  &3\\
			HS42  &4 &2  &&11  &9  &&3  &3\\
			HS46  &5 &2  &&29  &27  &&16  &16\\
			HS48  &5 &2  &&13  &10  &&3  &3\\
			HS49  &5 &2  &&27  &22  &&16  &16\\
			HS50  &5 &3  &&25  &15  &&8  &8\\
			HS51  &5 &3  &&10  &9  &&3  &2\\
			HS52  &5 &3  &&8  &7  &&3  &2\\
			HS61  &3 &2  &&13  &11  &&6  &6\\
			HS77  &5 &2  &&29  &26  &&10  &10\\
			HS78  &5 &3  &&9  &9  &&8  &8\\
			HS79  &5 &3  &&13  &13  &&5  &5\\
			HS100LNP  &7 &2  &&79  &35  &&8  &8\\
			HYPCIR  &2 &2  &&6  &5  &&5  &4\\
			INTEGREQ  &102 &100  &&2  &2  &&3  &3\\
			MARATOS  &2 &1  &&5  &5  &&4  &4\\
			METHANB8  &31 &31  &&3  &3  &&4  &3\\
			ORTHREGB  &27 &6  &&7  &7  &&3  &3\\
			RECIPE  &2 &2  &&12  &12  &&3  &2 \\
			RSNBRNE  &2 &2  &&36  &10  &&13  &12\\
			S316-322  &2 &1  &&9  &8  &&5  &5\\
			SINVALNE  &2 &2  &&35  &10  &&3  &2\\
			TRIGGER  &7 &6  &&8  &8  &&8  &7 \\
			ZANGWIL3  &3 &3  &&2  &2  &&2  &1\\ 			
			\hline
			Sum  &  &   &&818  &610  &&334  &322\\ 			
			\hline
		\end{longtable}
	}
\end{center}

\begin{center}
{
\begin{longtable}{cccccllcllll}
\caption{\textcolor{blue}{Comparison results with the variant of Algorithm 2.1 \cite{Chen2020COAP} by using objective function}}\label{tabchenobj}\\
\endfirsthead
\multicolumn{10}{l}{Table \ref{tabchenobj} continued}\\
\hline
\multirow{2}{*}{$\rho$} & \multirow{2}{*}{$t$}  && \multicolumn{3}{c}{Algorithm 2.1 \cite{Chen2020COAP}} &\multirow{2}{*}{ } & \multicolumn{3}{c}{$\mathrm{SSARC_{q}K}$}\\
\cmidrule{4-6}
\cmidrule{8-10}
&& &$N_{it}$&$N_f$&Res & &$N_{it}$&$N_f$&Res \\
\hline
\endhead
\hline
\multirow{2}{*}{$\rho$} & \multirow{2}{*}{$t$}  && \multicolumn{3}{c}{Algorithm 2.1 \cite{Chen2020COAP}} && \multicolumn{3}{c}{$\mathrm{SSARC_{q}K}$}\\
\cmidrule{4-6}
\cmidrule{8-10}
&  & &$N_{it}$&$N_f$&Res  & &$N_{it}$&$N_f$&Res \\
\hline
 2         &$\pm10^{-1}$  &&7   &12  &1.2000e-20 &&7 &7 &5.8387e-12 \\
           &$\pm10^{-2}$  &&4   &10  &5.7000e-12 &&7 &7 &6.2791e-13 \\
           &$\pm10^{-3}$  &&5   &15  &8.8000e-16 &&7 &7 &6.2836e-14 \\
           &$\pm10^{-4}$  &&7   &32  &4.3000e-11 &&7 &7 &6.2836e-15 \\
           &$\pm10^{-5}$  &&3   &18  &5.1000e-12 &&7 &7 &6.2836e-16 \\
           & Sum          &&26  &87  &           &&35&35&  \\
  \cmidrule{2-10}
 10        &$\pm10^{-1}$  &&27  &33  &7.5535e-12 &&58&59&9.1573e-11 \\
           &$\pm10^{-2}$  &&11  &18  &2.7445e-13 &&34&34&3.0102e-11 \\
           &$\pm10^{-3}$  &&17  &54  &2.5558e-22 &&34&35&2.0186e-11 \\
           &$\pm10^{-4}$  &&26  &133 &5.9308e-12 &&27&28&2.0778e-11 \\
           &$\pm10^{-5}$  &&11  &68  &6.6169e-14 &&26&27&9.6110e-11 \\
           & Sum          &&92  &306 &           &&179&183&  \\
  \cmidrule{2-10}                                                   
 100       &$\pm10^{-1}$  &&108 &126 &1.2644e-23 &&F &F &4.9581e-06 \\
           &$\pm10^{-2}$  &&167 &517 &2.1458e-12 &&F &F &1.4364e-06 \\
           &$\pm10^{-3}$  &&F   &F   &2.1884e-04 &&F &F &9.0358e-07 \\
           &$\pm10^{-4}$  &&103 &620 &1.7097e-13 &&F &F &2.1818e-06 \\
           &$\pm10^{-5}$  &&F   &F   &1.0992e-07 &&F &F &3.9074e-06 \\
           & Sum          &&378+2F&1263+2F&      &&5F&5F&  \\                                                          
\hline
\end{longtable}
}
\end{center}

\begin{center}
{
\begin{longtable}{cccccllcllll}
\caption{\textcolor{blue}{Comparison results with the variant of Algorithm 2.1 in \cite{Chen2020COAP} by using Lagrangian function}}\label{tabchenlag}\\
\endfirsthead
\multicolumn{10}{l}{Table \ref{tabchenlag} continued}\\
\hline
\multirow{2}{*}{$\rho$} & \multirow{2}{*}{$t$}  && \multicolumn{3}{c}{Algorithm 2.1 \cite{Chen2020COAP}} &\multirow{2}{*}{ } & \multicolumn{3}{c}{$\mathrm{SSARC_{q}K}$}\\
\cmidrule{4-6}
\cmidrule{8-10}
&& &$N_{it}$&$N_f$&Res & &$N_{it}$&$N_f$&Res \\
\hline
\endhead
\hline
\multirow{2}{*}{$\rho$} & \multirow{2}{*}{$t$}  && \multicolumn{3}{c}{Algorithm 2.1 \cite{Chen2020COAP}} && \multicolumn{3}{c}{$\mathrm{SSARC_{q}K}$}\\
\cmidrule{4-6}
\cmidrule{8-10}
&  & &$N_{it}$&$N_f$&Res  & &$N_{it}$&$N_f$&Res \\
\hline    
 2         &$\pm10^{-1}$  &&4 &5 &2.7200e-34 &&5 &5 &4.2130e-15 \\
           &$\pm10^{-2}$  &&3 &4 &3.2500e-31 &&4 &4 &4.6814e-17 \\
           &$\pm10^{-3}$  &&2 &3 &2.5000e-13 &&3 &3 &2.5002e-13  \\
           &$\pm10^{-4}$  &&2 &3 &6.7700e-28 &&3 &3 &9.9841e-16  \\
           &$\pm10^{-5}$  &&2 &3 &0.0000e+00 &&2 &2 &1.0000e-10  \\ 
           & Sum          &&13&18&           &&17&17&            \\
\cmidrule{2-10}                                                            
 10        &$\pm10^{-1}$  &&4 &5 &5.9346e-33 &&5 &5 &4.2130e-15 \\
           &$\pm10^{-2}$  &&3 &4 &1.7347e-30 &&4 &4 &1.2192e-14 \\
           &$\pm10^{-3}$  &&2 &3 &2.5002e-13 &&3 &3 &2.6606e-12 \\
           &$\pm10^{-4}$  &&2 &3 &1.0164e-27 &&3 &3 &3.1570e-14 \\
           &$\pm10^{-5}$  &&2 &3 &3.8118e-30 &&2 &2 &9.9999e-11 \\  
           & Sum          &&13&18&           &&17&17&            \\
\cmidrule{2-10}                                                               
 100       &$\pm10^{-1}$  &&4 &5 &3.9013e-32 &&5 &5 &1.4371e-16 \\
           &$\pm10^{-2}$  &&3 &4 &4.6620e-29 &&4 &4 &1.2192e-14 \\
           &$\pm10^{-3}$  &&2 &3 &2.5002e-13 &&3 &3 &2.6606e-12 \\
           &$\pm10^{-4}$  &&2 &3 &5.8276e-27 &&3 &3 &3.1570e-14 \\
           &$\pm10^{-5}$  &&2 &3 &5.4940e-29 &&2 &2 &9.9999e-11 \\  
           & Sum          &&13&18&           &&17&17&            \\
\cmidrule{2-10}                                                               
 1000      &$\pm10^{-1}$  &&4 &5 &2.2068e-31 &&5 &5 &4.2130e-15 \\
           &$\pm10^{-2}$  &&3 &4 &4.1973e-28 &&4 &4 &1.2192e-14  \\
           &$\pm10^{-3}$  &&2 &3 &2.5002e-13 &&3 &3 &2.6606e-12  \\
           &$\pm10^{-4}$  &&2 &3 &1.0937e-25 &&3 &3 &3.1570e-14  \\
           &$\pm10^{-5}$  &&2 &3 &1.1668e-28 &&2 &2 &9.9999e-11 \\ 
            & Sum          &&13&18&           &&17&17&            \\
\hline
\end{longtable}
}
\end{center}

\begin{center}
{
\begin{longtable}{lcccclllclll}
\caption{\textcolor{blue}{Comparison results with Algorithm 2.1 in \cite{Chen2020COAP}}}\label{tab2020chen}\\
\endfirsthead
\multicolumn{10}{l}{Table \ref{tab2020chen} continued}\\
\hline
\multirow{2}{*}{Problem} & \multicolumn{2}{c}{Dimension}  && \multicolumn{3}{c}{Algorithm 2.1 \cite{Chen2020COAP}} &\multirow{2}{*}{ } & \multicolumn{2}{c}{$\mathrm{SSARC_{q}K}$}\\
\cmidrule{2-3}
\cmidrule{5-7}
\cmidrule{9-11}
&$n$&$p$  & &$N_f$&$N_c$&Res & &$N_f$,$N_c$&Res \\
\hline
\endhead
\hline
\multirow{2}{*}{Problem} & \multicolumn{2}{c}{Dimension}  && \multicolumn{3}{c}{Algorithm 2.1 \cite{Chen2020COAP}} && \multicolumn{2}{c}{$\mathrm{SSARC_{q}K}$}\\
\cmidrule{2-3}
\cmidrule{5-7}
\cmidrule{9-11}
&$n$&$p$  & &$N_f$&$N_c$&Res & &$N_f,N_c$&Res \\
\hline
 AIRCRFTA  &8    &5    &&12  &21  &4.5003e-06 &&4  &2.2215e-16 \\
 ARGTRIG   &200  &200  &&4   &4   &8.1520e-10 &&4  &8.1517e-10 \\
 ARTIF     &502  &500  &&9   &15  &4.6233e-06 &&7  &1.4900e-09 \\
 AUG2D     &703  &300  &&4   &4   &1.6444e-09 &&6  &2.4374e-11 \\
 AUG2DC    &703  &300  &&4   &4   &1.6444e-09 &&12 &1.2970e-09  \\
 BDVALUE   &102  &100  &&4   &6   &2.2861e-06 &&3  &7.3096e-09 \\
 BDVALUES  &102  &100  &&20  &25  &4.9060e-06 &&13 &2.2229e-10 \\
 BOOTH     &2    &2    &&2   &2   &1.6997e-10 &&3  &4.4409e-16 \\
 BRATU2D   &484  &400  &&6   &10  &7.1482e-06 &&3  &7.4054e-09 \\
 BRATU2DT  &484  &400  &&6   &10  &9.0672e-06 &&4  &3.6233e-14 \\
 BRATU3D   &125  &27   &&10  &18  &3.1612e-06 &&4  &1.6127e-13 \\
 BROWNALE  &200  &200  &&7   &7   &1.1686e-09 &&9  &2.0178e-10 \\
 BROYDN3D  &500  &500  &&14  &14  &$I $       &&5  &1.0741e-09 \\
 BT1       &2    &1    &&8   &9   &5.4160e-09 &&6  &7.1663e-09 \\
 BT2       &3    &1    &&11  &11  &4.7125e-08 &&12 &9.0428e-14 \\
 BT3       &5    &3    &&6   &6   &1.0082e-09 &&5  &8.9598e-10 \\
 BT4       &3    &2    &&9   &9   &2.2185e-07 &&7  &7.2720e-14 \\
 BT5       &3    &2    &&5   &5   &3.1301e-08 &&5  &8.9746e-09 \\
 BT6       &5    &2    &&14  &14  &2.3299e-08 &&11 &5.4852e-09 \\
 BT7       &5    &3    &&16  &18  &5.5488e-06 &&13 &3.0271e-14 \\
 BT8       &5    &2    &&12  &19  &3.2450e-11 &&15 &5.2684e-09 \\
 BT9       &4    &2    &&13  &15  &2.3820e-07 &&9  &5.2087e-13 \\
 BT10      &2    &2    &&7   &7   &4.4039e-09 &&7  &4.4039e-09 \\
 BT11      &5    &3    &&8   &8   &7.9708e-09 &&8  &1.4334e-11 \\
 BT12      &5    &3    &&7   &8   &3.3889e-07 &&5  &2.3846e-13 \\
 BYRDSPHR  &3    &2    &&10  &13  &3.8304e-11 &&9  &1.6387e-13 \\
 CBRATU2D  &98   &50   &&9   &16  &4.2945e-06 &&4  &1.5569e-15 \\
 CBRATU3D  &128  &16   &&4   &4   &3.8108e-12 &&4  &3.8684e-13 \\
 CHANDHEU  &100  &100  &&11  &11  &6.3736e-06 &&16 &9.5858e-09 \\
 CLUSTER   &2    &2    &&8   &8   &7.0656e-06 &&10 &9.7423e-12 \\
 CUBENE    &2    &2    &&14  &18  &2.8906e-12 &&6  &0.0000e+0 \\
 DECONVNE  &61   &40   &&2   &2   &1.3050e-10 &&3  &1.2385e-15 \\
 DTOC1L    &598  &396  &&3   &3   &4.2487e-14 &&10 &1.6970e-09 \\
 DTOC2     &598  &396  &&17  &27  &2.5577e-06 &&8  &1.5442e-10 \\
 DTOC3     &299  &198  &&4   &4   &8.8177e-09 &&6  &2.8069e-10 \\
 DTOC4     &299  &198  &&3   &3   &2.6114e-07 &&3  &4.9860e-10 \\
 DTOC5     &19   &9    &&4   &4   &2.2816e-08 &&3  &3.9108e-09 \\
 EIGENA2   &6    &3    &&3   &3   &2.4252e-12 &&5  &1.1494e-10 \\
 EIGENAU   &110  &110  &&7   &7   &3.8718e-10 &&7  &1.4071e-12 \\
 EIGENB2   &6    &3    &&9   &9   &1.2406e-06 &&15 &1.6610e-10 \\
 EIGENC2   &30   &15   &&11  &11  &2.3420e-08 &&11 &6.0021e-09 \\
 FLOSP2TL  &363  &323  &&9   &9   &3.6044e-06 &&6  &5.3968e-11 \\
 GENHS28   &10   &8    &&3   &3   &1.2483e-14 &&4  &6.4625e-11 \\
 GOTTFR    &2    &2    &&9   &14  &2.1680e-10 &&8  &1.1392e-09 \\
 GRIDNETB  &180  &100  &&27  &27  &8.4083e-06 &&9  &1.0438e-09 \\
 GRIDNETE  &180  &100  &&20  &20  &7.5909e-06 &&9  &5.8997e-10 \\
 HAGER1    &1001 &500  &&2   &2   &3.6216e-11 &&3  &9.2810e-09 \\
 HAGER2    &1001 &500  &&6   &6   &1.3610e-09 &&4  &1.5798e-09 \\
 HAGER3    &1001 &500  &&14  &14  &8.8200e-06 &&4  &1.4325e-09 \\
 HATFLDF   &3    &3    &&28  &43  &3.0032e-09 &&7  &2.8311e-12 \\
 HATFLDG   &25   &25   &&8   &9   &4.8314e-06 &&9  &4.7181e-15 \\
 HEART6    &6    &6    &&72  &77  &6.2200e-07 &&29 &6.2690e-09 \\
 HEART8    &8    &8    &&48  &66  &3.5184e-09 &&12 &2.1299e-14 \\
 HIMMELBA  &2    &2    &&3   &3   &1.4965e-10 &&2  &0.0000e+0 \\
 HIMMELBC  &2    &2    &&5   &5   &6.8086e-06 &&6  &0.0000e+0 \\
 HIMMELBE  &3    &3    &&3   &3   &6.2354e-11 &&3  &0.0000e+0 \\
 HS6       &2    &1    &&10  &13  &3.5271e-10 &&8  &5.9686e-13 \\
 HS7       &2    &1    &&9   &11  &7.0501e-11 &&8  &1.5913e-10 \\
 HS8       &2    &2    &&5   &5   &7.8113e-08 &&5  &2.2061e-12 \\
 HS9       &2    &1    &&4   &4   &6.5239e-08 &&7  &1.7298e-09 \\
 HS26      &3    &1    &&19  &19  &6.4478e-06 &&17 &6.9605e-09 \\
 HS27      &3    &1    &&23  &25  &9.9924e-06 &&11 &7.4560e-09 \\
 HS28      &3    &1    &&2   &2   &1.7764e-15 &&4  &7.0407e-12 \\
 HS39      &4    &2    &&13  &15  &2.3820e-07 &&9  &5.6164e-11 \\
 HS40      &4    &3    &&4   &4   &8.6843e-10 &&4  &3.4209e-10 \\
 HS42      &4    &2    &&4   &4   &5.2056e-09 &&4  &9.3173e-09 \\
 HS46      &5    &2    &&17  &17  &5.4805e-06 &&23 &7.6893e-09 \\
 HS47      &5    &3    &&18  &18  &3.2632e-06 &&20 &5.4288e-09 \\
 HS48      &5    &2    &&3   &3   &1.7311e-15 &&4  &8.6179e-13 \\
 HS49      &5    &2    &&15  &15  &4.0666e-06 &&22 &8.3110e-09 \\
 HS50      &5    &3    &&9   &9   &4.4177e-12 &&10 &3.5567e-13 \\
 HS51      &5    &3    &&2   &2   &3.1426e-15 &&3  &4.3252e-09 \\
 HS52      &5    &3    &&3   &3   &2.5670e-10 &&3  &3.4710e-09 \\
 HS56      &7    &4    &&9   &9   &3.5344e-06 &&13 &3.5278e-09 \\
 HS61      &3    &2    &&6   &6   &3.1253e-10 &&7  &1.3609e-11 \\
 HS77      &5    &2    &&16  &16  &7.9904e-09 &&11 &8.9764e-10 \\
 HS78      &5    &3    &&5   &5   &9.7961e-11 &&11 &2.8086e-09 \\
 HS79      &5    &3    &&5   &5   &4.5896e-09 &&6  &1.1567e-11 \\
 HS100LNP  &7    &2    &&8   &10  &8.4215e-06 &&9  &2.4747e-12 \\
 HS111LNP  &10   &3    &&12  &13  &2.7335e-07 &&13 &1.3675e-10 \\
 HYDCAR20  &99   &99   &&223 &235 &9.3669e-06 &&10 &1.3560e-11 \\
 HYDCAR6   &29   &29   &&73  &84  &1.6698e-10 &&6  &5.9618e-14 \\
 HYPCIR    &2    &2    &&5   &6   &1.6698e-10 &&5  &1.9989e-10 \\
 INTEGREQ  &102  &100  &&2   &2   &1.8756e-10 &&4  &3.9141e-14 \\
 JUNKTURN  &510  &350  &&25  &34  &4.7652e-09 &&11 &1.3013e-09 \\
 MARATOS   &2    &1    &&4   &4   &1.5854e-08 &&4  &1.7235e-09 \\
 METHANB8  &31   &31   &&3   &3   &1.2763e-07 &&4  &1.9401e-13 \\
 METHANL8  &31   &31   &&8   &8   &9.0804e-08 &&5  &5.2325e-11 \\
 MSQRTA    &100  &100  &&30  &34  &4.9768e-09 &&7  &2.5112e-14 \\
 MSQRTB    &100  &100  &&44  &51  &2.1423e-09 &&7  &2.5602e-14 \\
 MWRIGHT   &5    &3    &&8   &8   &1.5122e-09 &&6  &3.6934e-11 \\
 OPTCTRL3  &302  &200  &&21  &21  &3.5950e-07 &&28 &6.4662e-09 \\
 OPTCTRL6  &302  &200  &&21  &21  &3.5950e-07 &&23 &3.1631e-09 \\
 ORTHREGA  &517  &256  &&16  &27  &4.5780e-06 &&30 &6.1036e-09 \\
 ORTHREGB  &27   &6    &&25  &30  &7.7097e-07 &&4  &9.3454e-11 \\
 ORTHREGC  &1005 &500  &&12  &12  &3.9966e-06 &&6  &2.6533e-09 \\
 POWELLBS  &2    &2    &&26  &30  &2.2353e-07 &&13 &7.9716e-09 \\
 POWELLSQ  &2    &2    &&12  &15  &9.2826e-08 &&28 &3.3565e-12 \\
 RECIPE    &3    &3    &&12  &12  &6.0225e-06 &&3  &8.9015e-16 \\
 RSNBRNE   &2    &2    &&18  &26  &1.1040e-12 &&13 &0.0000e+0 \\
 S316-322  &2    &1    &&5   &5   &5.9127e-07 &&6  &5.0306e-12 \\
 SINVALNE  &2    &2    &&16  &20  &3.0518e-12 &&3  &5.4133e-15 \\
 TRIGGER   &7    &6    &&13  &24  &3.0015e-06 &&8  &6.2235e-10 \\
 ZANGWIL3  &3    &3    &&6   &6   &4.9115e-09 &&2   &6.1535e-15 \\
\hline
 Sum  &     &    &&1443   &1664   &  &&883   &  \\
\hline
\end{longtable}
}
\end{center}
\bibliography{cites}   

@article {Onwunta2024JSC,
    AUTHOR = {Onwunta, Akwum and Royer, Cl\'ement W.},
     TITLE = {Complexity analysis of regularization methods for implicitly
              constrained least squares},
   JOURNAL = {J. Sci. Comput.},
  FJOURNAL = {Journal of Scientific Computing},
    VOLUME = {101},
      YEAR = {2024},
    NUMBER = {3},
     PAGES = {Paper No. 54, 22},
      ISSN = {0885-7474,1573-7691},
}

@article {Fischer2024COAP,
    AUTHOR = {Fischer, Andreas and Izmailov, Alexey F. and Solodov, Mikhail
              V.},
     TITLE = {The {L}evenberg-{M}arquardt method: an overview of modern
              convergence theories and more},
   JOURNAL = {Comput. Optim. Appl.},
  FJOURNAL = {Computational Optimization and Applications. An International
              Journal},
    VOLUME = {89},
      YEAR = {2024},
    NUMBER = {1},
     PAGES = {33--67},
      ISSN = {0926-6003,1573-2894},
}

@article {Agarwal1,
    AUTHOR = {Agarwal, Naman and Boumal, Nicolas and Bullins, Brian and
              Cartis, Coralia},
     TITLE = {Adaptive regularization with cubics on manifolds},
   JOURNAL = {Math. Program.},
  FJOURNAL = {Mathematical Programming},
    VOLUME = {188},
      YEAR = {2021}, 
    NUMBER = {1},
     PAGES = {85--134},
}

@article {Bergou4,
    AUTHOR = {Bergou, El Houcine and Diouane, Youssef and Gratton, Serge},
     TITLE = {A line-search algorithm inspired by the adaptive cubic  regularization framework and complexity analysis},
   JOURNAL = {J. Optim. Theory Appl.},
  FJOURNAL = {Journal of Optimization Theory and Applications},
    VOLUME = {178},
      YEAR = {2018},
    NUMBER = {3},
     PAGES = {885--913},
}

@article {Bianconcini5,
    AUTHOR = {Bianconcini, Tommaso and Liuzzi, Giampaolo and Morini,
              Benedetta and Sciandrone, Marco},
     TITLE = {On the use of iterative methods in cubic regularization for
              unconstrained optimization},
   JOURNAL = {Comput. Optim. Appl.},
  FJOURNAL = {Computational Optimization and Applications. An International
              Journal},
    VOLUME = {60},
      YEAR = {2015},
    NUMBER = {1},
     PAGES = {35--57},
}

@article {Bianconcini6,
    AUTHOR = {Bianconcini, Tommaso and Sciandrone, Marco},
     TITLE = {A cubic regularization algorithm for unconstrained
              optimization using line search and nonmonotone techniques},
   JOURNAL = {Optim. Methods Softw.},
  FJOURNAL = {Optimization Methods \& Software},
    VOLUME = {31},
      YEAR = {2016},
    NUMBER = {5},
     PAGES = {1008--1035},
}

@article {Birgin7,
    AUTHOR = {Birgin, E. G. and Gardenghi, J. L. and Mart\'inez, J. M. and
              Santos, S. A.},
     TITLE = {On the use of third-order models with fourth-order
              regularization for unconstrained optimization},
   JOURNAL = {Optim. Lett.},
  FJOURNAL = {Optimization Letters},
    VOLUME = {14},
      YEAR = {2020},
    NUMBER = {4},
     PAGES = {815--838},
}

@article {Cartis11,
    AUTHOR = {Cartis, Coralia and Gould, Nicholas I. M. and Toint, Philippe
              L.},
     TITLE = {Adaptive cubic regularisation methods for unconstrained
              optimization. {P}art {I}: motivation, convergence and
              numerical results},
   JOURNAL = {Math. Program.},
  FJOURNAL = {Mathematical Programming},
    VOLUME = {127},
      YEAR = {2011},
    NUMBER = {2},
     PAGES = {245--295},
}

@article {Cartis12,
    AUTHOR = {Cartis, Coralia and Gould, Nicholas I. M. and Toint, Philippe
              L.},
     TITLE = {Adaptive cubic regularisation methods for unconstrained
              optimization. {P}art {II}: worst-case function- and
              derivative-evaluation complexity},
   JOURNAL = {Math. Program.},
  FJOURNAL = {Mathematical Programming},
    VOLUME = {130},
      YEAR = {2011},
    NUMBER = {2},
     PAGES = {295--319},
}

@book {Conn15,
    AUTHOR = {Conn, Andrew R. and Gould, Nicholas I. M. and Toint, Philippe
              L.},
     TITLE = {Trust-region methods},
    SERIES = {MPS/SIAM Series on Optimization},
 PUBLISHER = {Society for Industrial and Applied Mathematics (SIAM),
              Philadelphia, PA; Mathematical Programming Society (MPS),
              Philadelphia, PA},
      YEAR = {2000},
}

@article {Dussault18,
    AUTHOR = {Dussault, Jean-Pierre},
     TITLE = {{${\rm ARC_q}$}: a new adaptive regularization by cubics},
   JOURNAL = {Optim. Methods Softw.},
  FJOURNAL = {Optimization Methods \& Software},
    VOLUME = {33},
      YEAR = {2018},
    NUMBER = {2},
     PAGES = {322--335},
}

@article {Dussault19,
    AUTHOR = {Dussault, Jean-Pierre and Migot, Tangi and Orban, Dominique},
     TITLE = {Scalable adaptive cubic regularization methods},
   JOURNAL = {Math. Program.},
  FJOURNAL = {Mathematical Programming},
    VOLUME = {207},
      YEAR = {2024},
    NUMBER = {1-2},
     PAGES = {191--225},
}

@article {Pei32,
    AUTHOR = {Pei, Yonggang and Song, Shaofang and Zhu, Detong},
     TITLE = {A filter sequential adaptive cubic regularization algorithm
              for nonlinear constrained optimization},
   JOURNAL = {Numer. Algorithms},
  FJOURNAL = {Numerical Algorithms},
    VOLUME = {93},
      YEAR = {2023},
    NUMBER = {4},
     PAGES = {1481--1507},
}

@article {Pei33,
    AUTHOR = {Pei, Yonggang and Song, Shaofang and Zhu, Detong},
     TITLE = {A sequential adaptive regularisation using cubics algorithm
              for solving nonlinear equality constrained optimization},
   JOURNAL = {Comput. Optim. Appl.},
  FJOURNAL = {Computational Optimization and Applications. An International
              Journal},
    VOLUME = {84},
      YEAR = {2023},
    NUMBER = {3},
     PAGES = {1005--1033},
}

@article {Vardi36,
    AUTHOR = {Vardi, Avi},
     TITLE = {A trust region algorithm for equality constrained
              minimization: convergence properties and implementation},
   JOURNAL = {SIAM J. Numer. Anal.},
  FJOURNAL = {SIAM Journal on Numerical Analysis},
    VOLUME = {22},
      YEAR = {1985},
    NUMBER = {3},
     PAGES = {575--591},
}

@article {Zhao40,
    AUTHOR = {Zhao, Ting and Liu, Hongwei and Liu, Zexian},
     TITLE = {New subspace minimization conjugate gradient methods based on
              regularization model for unconstrained optimization},
   JOURNAL = {Numer. Algorithms},
  FJOURNAL = {Numerical Algorithms},
    VOLUME = {87},
      YEAR = {2021},
    NUMBER = {4},
     PAGES = {1501--1534},
}

@article {Bellavia41,
    AUTHOR = {Bellavia, Stefania and Gurioli, Gianmarco and Morini,
              Benedetta},
     TITLE = {Adaptive cubic regularization methods with dynamic inexact
              {H}essian information and applications to finite-sum
              minimization},
   JOURNAL = {IMA J. Numer. Anal.},
  FJOURNAL = {IMA Journal of Numerical Analysis},
    VOLUME = {41},
      YEAR = {2021},
    NUMBER = {1},
     PAGES = {764--799},
}

@article {Benson42,
    AUTHOR = {Benson, Hande Y. and Shanno, David F.},
     TITLE = {Cubic regularization in symmetric rank-1 quasi-{N}ewton methods},
   JOURNAL = {Math. Program. Comput.},
  FJOURNAL = {Mathematical Programming Computation},
    VOLUME = {10},
      YEAR = {2018},
    NUMBER = {4},
     PAGES = {457--486},
}

@article {Cartis43,
    AUTHOR = {Cartis, C. and Gould, N. I. M. and Toint, Ph. L.},
     TITLE = {A concise second-order complexity analysis for unconstrained
              optimization using high-order regularized models},
   JOURNAL = {Optim. Methods Softw.},
  FJOURNAL = {Optimization Methods \& Software},
    VOLUME = {35},
      YEAR = {2020},
    NUMBER = {2},
     PAGES = {243--256},
}

@article {Dehghan44,
    AUTHOR = {Dehghan Niri, T. and Heydari, M. and Hosseini, M. M.},
     TITLE = {An improvement of adaptive cubic regularization method for
              unconstrained optimization problems},
   JOURNAL = {Int. J. Comput. Math.},
  FJOURNAL = {International Journal of Computer Mathematics},
    VOLUME = {98},
      YEAR = {2021},
    NUMBER = {2},
     PAGES = {271--287},
}

@book {Dennis45,
    AUTHOR = {Dennis, Jr., J. E. and Schnabel, Robert B.},
     TITLE = {Numerical methods for unconstrained optimization and nonlinear
              equations},
    SERIES = {Classics in Applied Mathematics},
 PUBLISHER = {Society for Industrial and Applied Mathematics (SIAM),
              Philadelphia, PA},
      YEAR = {1996},
}

@article {Cartis46,
    AUTHOR = {Cartis, C. and Gould, N. I. M. and Toint, Ph.\ L.},
     TITLE = {An adaptive cubic regularization algorithm for nonconvex
              optimization with convex constraints and its
              function-evaluation complexity},
   JOURNAL = {IMA J. Numer. Anal.},
  FJOURNAL = {IMA Journal of Numerical Analysis},
    VOLUME = {32},
      YEAR = {2012},
    NUMBER = {4},
     PAGES = {1662--1695},
}

@article {Zhu47,
    AUTHOR = {Huang, Xiaojin and Zhu, Detong},
     TITLE = {An interior affine scaling cubic regularization algorithm for
              derivative-free optimization subject to bound constraints},
   JOURNAL = {J. Comput. Appl. Math.},
  FJOURNAL = {Journal of Computational and Applied Mathematics},
    VOLUME = {321},
      YEAR = {2017},
     PAGES = {108--127},
}

@article{Gould2015cutest,
  title={CUTEst: a constrained and unconstrained testing environment with safe threads for mathematical optimization},
  author={Gould, Nicholas IM and Orban, Dominique and Toint, Philippe L},
  journal={Computational optimization and applications},
  volume={60},
  pages={545--557},
  year={2015},
}

@article {Chen2020COAP,
    AUTHOR = {Chen, Zhongwen and Dai, Yu-Hong and Liu, Jiangyan},
     TITLE = {A penalty-free method with superlinear convergence for
              equality constrained optimization},
   JOURNAL = {Comput. Optim. Appl.},
  FJOURNAL = {Computational Optimization and Applications. An International
              Journal},
    VOLUME = {76},
      YEAR = {2020},
    NUMBER = {3},
     PAGES = {801--833},
}

@article {PowellSIAMReview1986,
    AUTHOR = {Powell, M. J. D.},
     TITLE = {Convergence properties of algorithms for nonlinear
              optimization},
   JOURNAL = {SIAM Rev.},
  FJOURNAL = {SIAM Review. A Publication of the Society for Industrial and
              Applied Mathematics},
    VOLUME = {28},
      YEAR = {1986},
    NUMBER = {4},
     PAGES = {487--500},
      ISSN = {0036-1445},
   MRCLASS = {49D37 (49D10 65K10 90C30)},
  MRNUMBER = {867680},
}

@article {2002JOTA,
    AUTHOR = {Li, Z. F. and Osborne, M. R. and Prvan, T.},
     TITLE = {Adaptive algorithm for constrained least-squares problems},
   JOURNAL = {J. Optim. Theory Appl.},
  FJOURNAL = {Journal of Optimization Theory and Applications},
    VOLUME = {114},
      YEAR = {2002},
    NUMBER = {2},
     PAGES = {423--441},
      ISSN = {0022-3239,1573-2878},
}

@article {Liu2011SIOPT,
    AUTHOR = {Liu, Xinwei and Yuan, Yaxiang},
     TITLE = {A sequential quadratic programming method without a penalty
              function or a filter for nonlinear equality constrained
              optimization},
   JOURNAL = {SIAM J. Optim.},
  FJOURNAL = {SIAM Journal on Optimization},
    VOLUME = {21},
      YEAR = {2011},
    NUMBER = {2},
     PAGES = {545--571},
      ISSN = {1052-6234,1095-7189},
}

@article {Dolan,
    AUTHOR = {Dolan, Elizabeth D. and Mor\'{e}, Jorge J.},
     TITLE = {Benchmarking optimization software with performance profiles},
   JOURNAL = {Math. Program.},
  FJOURNAL = {Mathematical Programming. A Publication of the Mathematical
              Programming Society},
    VOLUME = {91},
      YEAR = {2002},
    NUMBER = {2, Ser. A},
     PAGES = {201--213},
}

@article {Byrd2010MP,
    AUTHOR = {Byrd, Richard H. and Curtis, Frank E. and Nocedal, Jorge},
     TITLE = {An inexact {N}ewton method for nonconvex equality constrained
              optimization},
   JOURNAL = {Math. Program.},
  FJOURNAL = {Mathematical Programming. A Publication of the Mathematical
              Programming Society},
    VOLUME = {122},
      YEAR = {2010},
    NUMBER = {2},
     PAGES = {273--299},
      ISSN = {0025-5610,1436-4646},
}

@article {Heinkenschloss2014SIOPT,
    AUTHOR = {Heinkenschloss, Matthias and Ridzal, Denis},
     TITLE = {A matrix-free trust-region {SQP} method for equality
              constrained optimization},
   JOURNAL = {SIAM J. Optim.},
  FJOURNAL = {SIAM Journal on Optimization},
    VOLUME = {24},
      YEAR = {2014},
    NUMBER = {3},
     PAGES = {1507--1541},
      ISSN = {1052-6234,1095-7189},
}

@article {Bergou2021SIAMJSC,
    AUTHOR = {Bergou, El Houcine and Diouane, Youssef and Kungurtsev,
              Vyacheslav and Royer, Cl\'ement W.},
     TITLE = {A nonmonotone matrix-free algorithm for nonlinear equality-constrained least-squares problems},
   JOURNAL = {SIAM J. Sci. Comput.},
  FJOURNAL = {SIAM Journal on Scientific Computing},
    VOLUME = {43},
      YEAR = {2021},
    NUMBER = {5},
     PAGES = {S743--S766},
      ISSN = {1064-8275,1095-7197},
}

@article {Birgin2017MP,
    AUTHOR = {Birgin, E. G. and Gardenghi, J. L. and Mart\'inez, J. M. and
              Santos, S. A. and Toint, Ph. L.},
     TITLE = {Worst-case evaluation complexity for unconstrained nonlinear
              optimization using high-order regularized models},
   JOURNAL = {Math. Program.},
  FJOURNAL = {Mathematical Programming},
    VOLUME = {163},
      YEAR = {2017},
    NUMBER = {1-2},
     PAGES = {359--368},
}

@book{Nocedal1999,
    AUTHOR = {Nocedal, Jorge and Wright, Stephen J.},
     TITLE = {Numerical optimization},
    SERIES = {Springer Series in Operations Research and Financial
              Engineering},
   EDITION = {Second},
 PUBLISHER = {Springer, New York},
      YEAR = {2006},
      ISBN = {978-0387-30303-1; 0-387-30303-0},
}

@book {sunyuanbook2006,
    AUTHOR = {Sun, Wenyu and Yuan, Ya-Xiang},
     TITLE = {Optimization theory and methods},
    SERIES = {Springer Optimization and Its Applications},
    VOLUME = {1},
      NOTE = {Nonlinear programming},
 PUBLISHER = {Springer, New York},
      YEAR = {2006},
}

@article {LiederSubproblem2020,
    AUTHOR = {Lieder, Felix},
     TITLE = {Solving large-scale cubic regularization by a generalized
              eigenvalue problem},
   JOURNAL = {SIAM J. Optim.},
  FJOURNAL = {SIAM Journal on Optimization},
    VOLUME = {30},
      YEAR = {2020},
    NUMBER = {4},
     PAGES = {3345--3358},
              }

@article {JiangRujun2021,
    AUTHOR = {Jiang, Rujun and Yue, Man-Chung and Zhou, Zhishuo},
     TITLE = {An accelerated first-order method with complexity analysis for
              solving cubic regularization subproblems},
   JOURNAL = {Comput. Optim. Appl.},
  FJOURNAL = {Computational Optimization and Applications. An International
              Journal},
    VOLUME = {79},
      YEAR = {2021},
    NUMBER = {2},
     PAGES = {471--506},
              }

@article {shen2022SMAA,
    AUTHOR = {Jia, Xiaojing and Liang, Xin and Shen, Chungen and Zhang,
              Lei-Hong},
     TITLE = {Solving the cubic regularization model by a nested
              restarting Lanczos method},
   JOURNAL = {SIAM J. Matrix Anal. Appl.},
  FJOURNAL = {SIAM Journal on Matrix Analysis and Applications},
    VOLUME = {43},
      YEAR = {2022},
    NUMBER = {2},
     PAGES = {812--839},
              }

@article {Hsia2017OMS,
    AUTHOR = {Hsia, Yong and Sheu, Ruey-Lin and Yuan, Ya-xiang},
     TITLE = {Theory and application of {$p$}-regularized subproblems for
              {$p>2$}},
   JOURNAL = {Optim. Methods Softw.},
  FJOURNAL = {Optimization Methods \& Software},
    VOLUME = {32},
      YEAR = {2017},
    NUMBER = {5},
     PAGES = {1059--1077},
}

@article {Dussault2019INFORMS,
    AUTHOR = {Dussault, Jean-Pierre},
     TITLE = {A unified efficient implementation of trust-region type
              algorithms for unconstrained optimization},
   JOURNAL = {INFOR Inf. Syst. Oper. Res.},
  FJOURNAL = {INFOR. Information Systems and Operational Research},
    VOLUME = {58},
      YEAR = {2020},
    NUMBER = {2},
     PAGES = {290--309},
}

@article {2020LiuHongweiNA,
    AUTHOR = {Zhao, Ting and Liu, Hongwei and Liu, Zexian},
     TITLE = {New subspace minimization conjugate gradient methods based on
              regularization model for unconstrained optimization},
   JOURNAL = {Numer. Algorithms},
  FJOURNAL = {Numerical Algorithms},
    VOLUME = {87},
      YEAR = {2021},
    NUMBER = {4},
     PAGES = {1501--1534},
      ISSN = {1017-1398},
   MRCLASS = {65K10 (90C26 90C53)},
}

@article {Park2020JOTA,
    AUTHOR = {Park, Seonho and Jung, Seung Hyun and Pardalos, Panos M.},
     TITLE = {Combining stochastic adaptive cubic regularization with
              negative curvature for nonconvex optimization},
   JOURNAL = {J. Optim. Theory Appl.},
  FJOURNAL = {Journal of Optimization Theory and Applications},
    VOLUME = {184},
      YEAR = {2020},
    NUMBER = {3},
     PAGES = {953--971},
      ISSN = {0022-3239},
   MRCLASS = {90C15 (49M15 65K10 90C06 90C60)},
  MRNUMBER = {4061676},
}

@article {Martnez2017JGO,
    AUTHOR = {Mart\'{\i}nez, J. M. and Raydan, M.},
     TITLE = {Cubic-regularization counterpart of a variable-norm trust-region method for unconstrained minimization},
   JOURNAL = {J. Global Optim.},
  FJOURNAL = {Journal of Global Optimization. An International Journal
              Dealing with Theoretical and Computational Aspects of Seeking
              Global Optima and Their Applications in Science, Management
              and Engineering},
    VOLUME = {68},
      YEAR = {2017},
    NUMBER = {2},
     PAGES = {367--385},
      ISSN = {0925-5001},
   MRCLASS = {90C26 (49M15 90C55)},
  MRNUMBER = {3649544},
MRREVIEWER = {Stefan M. Stefanov},
}

@article {Dennis1997SIOPT,
    AUTHOR = {Dennis, Jr., J. E. and El-Alem, Mahmoud and Maciel, Maria C.},
     TITLE = {A global convergence theory for general trust-region-based
              algorithms for equality constrained optimization},
   JOURNAL = {SIAM J. Optim.},
  FJOURNAL = {SIAM Journal on Optimization},
    VOLUME = {7},
      YEAR = {1997},
    NUMBER = {1},
     PAGES = {177--207},
      ISSN = {1052-6234},
}

@phdthesis{Wilson1963SQP,
  author    = {Wilson, Robert B.},
  title     = {A Simplicial Algorithm for Concave Programming},
  school    = {Harvard Business School},
  year      = {1963},
}

@article {Han1976MP,
    AUTHOR = {Han, Shih Ping},
     TITLE = {Superlinearly convergent variable metric algorithms for
              general nonlinear programming problems},
   JOURNAL = {Math. Programming},
  FJOURNAL = {Mathematical Programming},
    VOLUME = {11},
      YEAR = {1976},
    NUMBER = {3},
     PAGES = {263--282},
}

@article {Byrd1987SIAMna,
    AUTHOR = {Byrd, Richard H. and Schnabel, Robert B. and Shultz, Gerald
              A.},
     TITLE = {A trust region algorithm for nonlinearly constrained
              optimization},
   JOURNAL = {SIAM J. Numer. Anal.},
  FJOURNAL = {SIAM Journal on Numerical Analysis},
    VOLUME = {24},
      YEAR = {1987},
    NUMBER = {5},
     PAGES = {1152--1170},
}

@incollection {Boggs1995ActaNumer,
    AUTHOR = {Boggs, Paul T. and Tolle, Jon W.},
     TITLE = {Sequential quadratic programming},
 BOOKTITLE = {Acta numerica, 1995},
    SERIES = {Acta Numer.},
     PAGES = {1--51},
 PUBLISHER = {Cambridge Univ. Press, Cambridge},
      YEAR = {1995},
}

@article {Gould2005ActaNumer,
    AUTHOR = {Gould, Nick and Orban, Dominique and Toint, Philippe},
     TITLE = {Numerical methods for large-scale nonlinear optimization},
   JOURNAL = {Acta Numer.},
  FJOURNAL = {Acta Numerica},
    VOLUME = {14},
      YEAR = {2005},
     PAGES = {299--361},
}

@article {Gill2005SIAMRev,
    AUTHOR = {Gill, Philip E. and Murray, Walter and Saunders, Michael A.},
     TITLE = {S{NOPT}: an {SQP} algorithm for large-scale constrained
              optimization},
   JOURNAL = {SIAM Rev.},
  FJOURNAL = {SIAM Review},
    VOLUME = {47},
      YEAR = {2005},
    NUMBER = {1},
     PAGES = {99--131},
}
\bibstyle{sn-aps}

\end{document}